\documentclass[reqno, backref, 11pt]{amsart}

\usepackage[x11names]{xcolor}
\usepackage{enumerate, amsmath, amsthm, amsfonts, amssymb,  mathrsfs, graphicx, paralist, fancyvrb, mathtools,url,textcomp}
\definecolor{persianred}{rgb}{0.8, 0.2, 0.2}
\definecolor{burgundy}{rgb}{0.5, 0.0, 0.13}
\definecolor{prussianblue}{rgb}{0.0, 0.19, 0.33}
\definecolor{aquamarine}{rgb}{0.5, 1.0, 0.83}
\definecolor{ao(english)}{rgb}{0.0, 0.5, 0.0}
\definecolor{blizzardblue}{rgb}{0.67, 0.9, 0.93}
\definecolor{blue(munsell)}{rgb}{0.0, 0.5, 0.69}
\usepackage[bookmarks, colorlinks=true, linkcolor=persianred, citecolor=ao(english), urlcolor=blue(munsell)]{hyperref}

\usepackage[all]{xy}
\usepackage{url}
\usepackage{amssymb}
\usepackage{hyperref}
\usepackage[margin=.94 in]{geometry}
\usepackage{color,soul}
\usepackage{amsthm}
\usepackage{marvosym}
\usepackage{upgreek}
\usepackage{mathtools}

\DeclareMathOperator{\Spec}{Spec}

\DeclareMathOperator{\Ass}{Ass}
\DeclareMathOperator{\Hom}{Hom}
\DeclareMathOperator{\im}{im}

\DeclareMathOperator{\Spl}{Spl}
\DeclareMathOperator{\Nor}{Nor}
\DeclareMathOperator{\Frac}{Frac}
\DeclareMathOperator{\eval}{eval}
\DeclareMathOperator{\colim}{colim}
\DeclareMathOperator{\id}{id}
\DeclareMathOperator{\s}{sep}
\DeclareMathOperator{\Tr}{Tr}
\DeclareMathOperator{\nil}{nil}
\DeclareMathOperator{\cl}{cl}

\DeclareMathOperator{\Mod}{Mod}
\DeclareMathOperator{\fg}{fg}

\theoremstyle{plain}
\newtheorem{theorem}[subsubsection]{Theorem}
\newtheorem{proposition}[subsubsection]{Proposition}
\newtheorem{lemma}[subsubsection]{Lemma}

\newtheorem{corollary}[subsubsection]{Corollary}

\theoremstyle{definition}
\newtheorem{definition}[subsubsection]{Definition}
\newtheorem{example}[subsubsection]{Example}

\newtheorem{remark}[subsubsection]{Remark}

\theoremstyle{plain}

\theoremstyle{definition}

\theoremstyle{plain}

\theoremstyle{plain}

\numberwithin{equation}{subsubsection}
\numberwithin{equation}{subsubsection}
\normalfont
\usepackage[T1]{fontenc}{}
\usepackage{setspace}

\usepackage{tikz-cd}

\makeatletter
\def\@tocline#1#2#3#4#5#6#7{\relax
  \ifnum #1>\c@tocdepth 
  \else
    \par \addpenalty\@secpenalty\addvspace{#2}%
    \begingroup \hyphenpenalty\@M
    \@ifempty{#4}{%
      \@tempdima\csname r@tocindent\number#1\endcsname\relax
    }{%
      \@tempdima#4\relax
    }%
    \parindent\z@ \leftskip#3\relax \advance\leftskip\@tempdima\relax
    \rightskip\@pnumwidth plus4em \parfillskip-\@pnumwidth
    #5\leavevmode\hskip-\@tempdima
      \ifcase #1
       \or\or \hskip 1em \or \hskip 2em \else \hskip 3em \fi%
      #6\nobreak\relax
    \dotfill\hbox to\@pnumwidth{\@tocpagenum{#7}}\par
    \nobreak
    \endgroup
  \fi}
\makeatother
\usepackage{hyperref}

\title{Openness of splinter loci in prime characteristic}
\author{Rankeya Datta}
\address{Department of Mathematics, University of Missouri, Columbia, MO 65211, USA}
\email{rankeya.datta@missouri.edu}
\author{Kevin Tucker\vspace{-2ex}}
\address{Department of Mathematics, University of Illinois at Chicago, Chicago, IL 60607, USA}
\email{kftucker@uic.edu}
\thanks{The first author was supported in part by an AMS-Simons travel grant, and the second author in part by NSF Grants DMS \#1602070, \#1707661, and \#1246844.}

\begin{document}

\maketitle

\begin{abstract}
{}A splinter 
is a notion of singularity that has seen numerous recent applications, especially 
in connection with the direct summand theorem, the mixed 
characteristic minimal model program, Cohen--Macaulayness of absolute integral 
closures and cohomology vanishing theorems. 
Nevertheless, many basic questions about these singularities remain elusive.
One outstanding problem is whether
the splinter property spreads from a point to an open neighborhood of a noetherian
scheme.
Our paper addresses this problem in prime characteristic, where
we show that  a locally noetherian scheme 
that has finite Frobenius or that is locally essentially of finite type over a 
quasi-excellent local ring has an open splinter locus. In particular,
all varieties over fields of positive characteristic have open splinter loci.
Intimate connections are established
between the openness of splinter loci and $F$-compatible ideals,
which are prime characteristic analogues of log canonical centers. 
We prove the surprising fact that for a large class of noetherian
rings with pure (aka universally injective) Frobenius, 
the splinter condition is detected by the
splitting of a single generically \'etale finite extension. 
We also show that for a noetherian $\textbf{N}$-graded ring over
a field, the homogeneous maximal ideal detects the splinter property. 
\end{abstract}

\section{Introduction}

A noetherian ring is a splinter if it is a direct summand
of every finite cover. A splinter is a notion of singularity
since we now know that regular rings satisfy this
property by the celebrated direct summand theorem 
\cite{Hoc73, And18, Bha18}. For any notion of singularity, or more
generally, a property $\mathcal{P}$ of noetherian local rings, it is natural to ask if 
\[
\textrm{$\{x \in X \colon \mathcal{O}_{X,x}$ has $\mathcal{P}\}$}
\]
is an open subset of a locally noetherian scheme $X$. 
Such openness of loci questions were perhaps first
considered systematically by Grothendieck in \cite{EGAIV_2}. Many
fundamental 
local properties such as $R_n$, $S_n$, reduced, normal, Gorenstein, complete
intersection, Cohen--Macaulay, 
among others, are known to have open loci for \emph{most} locally noetherian
schemes that one encounters in arithmetic or geometry \cite{EGAIV_2, GM78, Val78}.

In this paper we consider the question of the openness of the splinter locus
of a locally noetherian scheme. As a preliminary observation, the splinter condition for
noetherian local $\mathbf{Q}$-algebras is equivalent to normality, and the
normal loci is open for locally noetherian schemes
that have open regular loci \cite[Cor.\ (6.13.5)]{EGAIV_2}. In particular,
the splinter locus of any quasi-excellent $\mathbf{Q}$-scheme is open because the 
normal locus of such a scheme is open. Our main 
result illustrates that a similar result holds 
for some large classes of 
locally noetherian schemes over $\mathbf{F}_p$.

\begin{theorem} {\normalfont (see Theorem \ref{thm:F-split-locus-open})}
\label{thm:main-thm}
Let $X$ be a scheme of prime characteristic $p > 0$ that satisfies any
of the following conditions:
\begin{enumerate}
    \item[(i)] $X$ is locally noetherian and $F$-finite.
    \item[(ii)] $X$ is locally essentially of finite type over a noetherian local ring
    $(A, \mathfrak m)$ of prime characteristic $p > 0$ with geometrically regular 
    formal fibers.
\end{enumerate}
Then $\textrm{$\{x \in X \colon \mathcal{O}_{X,x}$ is a splinter$\}$}$
is open in $X$.
\end{theorem}

\noindent In particular, the splinter locus of any scheme of finite type over a 
field, or more generally, a complete local ring of positive prime characteristic is 
open.

Showing the openness of the splinter locus
has proved to be a challenging problem in prime and mixed characteristics 
because 
splinters are far more mysterious away from equal characteristic $0$. For example, 
in prime and mixed characteristics,
splinters surprisingly
coincide with a derived counterpart called a derived splinter \cite{Bha12, Bha20},
and excellent local splinters are Cohen--Macaulay 
and pseudo-rational \cite{Smi94, Smi97(a), Bha20} (in prime characteristic
they are also $F$-rational \cite{Smi94}). Furthermore, prime characteristic splinters are 
conjecturally equivalent to $F$-regular singularities \cite{Sin99(b), CEMS18}, which are 
analogues of Kawamata log terminal singularities that arise in the birational 
classification of algebraic varieties over the complex numbers \cite{Smi97(a),
MS97, Har98}. This last conjectural equivalence implies 
one of the outstanding problems in tight closure theory,
namely that weak $F$-regularity is preserved under localization.
The notion of a splinter globalizes, and in a non-affine setting in mixed and 
positive characteristics they
have been recently called \emph{globally $+$-regular schemes}  \cite{BMPSTWW20}. 
This is in part because of their similarities
with globally $F$-regular varieties in positive characteristic \cite{Smi00, SS10}.

By working on an affine cover, the openness of loci for a local property reduces 
to a question about affine schemes, and 
Theorem \ref{thm:main-thm} follows from the following, more refined,
affine 
result.

\begin{theorem}{\normalfont (see Theorem \ref{thm:splinter-F-split})}
\label{thm:affine-main-thm}
Let $R$ be a noetherian $F$-pure domain of prime characteristic
$p > 0$ and assume that $R$ satisfies any of the following 
conditions:
\begin{enumerate}
    \item[(i)] $R$ is $F$-finite.
    \item[(ii)] $R$ is local and Frobenius split.
    \item[(iii)] $(A, \mathfrak m)$ is a noetherian local ring of 
    prime characteristic $p > 0$ with geometrically regular formal
    fibers and $R$ is essentially of finite type over $A$.
\end{enumerate}
Let $\mathcal{C}$ be the collection of finite $R$-subalgebras
of $R^+$ and for an $R$-algebra $S$, let
\[
\uptau_{S/R} \coloneqq \im(\Hom_R(S,R) \xrightarrow{\eval \MVAt 1}
R).
\]
Then we have the following:
\begin{enumerate}
    \item $\{\uptau_{S/R} \colon S \in \mathcal{C}\}$ is a finite
    set of radical ideals of $R$.
    \item The splinter locus of $\Spec(R)$ is open and its 
    complement is $\mathbf{V}(\uptau_R)$, where
    $$\uptau_R \coloneqq \bigcap_{S \in \mathcal C} \uptau_{S/R}.$$
    \item There exists $S \in \mathcal{C}$ such that 
    $R \subseteq S$ is generically \'etale and if 
    $R \hookrightarrow S$ splits, then $R$ is a splinter.
\end{enumerate}
\end{theorem}

\noindent Here $R^+$ deonotes the absolute integral closure of 
$R$, that is, $R^+$ is the integral closure of $R$ in an algebraic
closure of its fraction field. 
$F$-purity is the universal injectivity of the Frobenius map on $R$.
It is a mild assumption when discussing questions
pertaining to splinters in prime characteristic because splinters
are automatically $F$-pure. A surprising aspect
of Theorem \ref{thm:affine-main-thm} is perhaps the fact that
for most
$F$-pure noetherian domains that arise in arithmetic and geometry,
the splinter property 
is determined by the splitting of a \emph{single} generically \'etale finite extension domain, even though the definition
of a splinter a priori requires the splitting of all finite
extension domains. Theorem \ref{thm:main-thm} is a formal consequence of Theorem \ref{thm:affine-main-thm} because the splinter locus is contained in the $F$-pure
locus of any locally noetherian $\mathbf{F}_p$-scheme, and the $F$-pure locus is known to be open
when $X$ satisfies the assumptions of Theorem \ref{thm:main-thm}.

The proof of Theorem \ref{thm:affine-main-thm}
is not particularly involved in the noetherian $F$-finite setting, 
so we briefly discuss our strategy for the expert.
The ideal $\uptau_{S/R}$ of Theorem 
\ref{thm:affine-main-thm} is called the \emph{trace of $S$ over $R$}, and the
content of the
Theorem is that these traces 
stabilize as $S$ ranges over the finite $R$-subalgebras of $R^+$ with 
appropriate purity assumptions on $R$. The stabilization of 
traces follows from the fact that trace ideals satisfy the
property of being \emph{uniformly $F$-compatible}, which is a prime characteristic analogue of the notion
of a center of log canonicity \cite{Sch09, Sch10, ST10}. 
For us, the key fact about uniformly $F$-compatible ideals 
is their finiteness under appropriate
assumptions. Namely,
Schwede showed that an $F$-pure noetherian $F$-finite ring
has only finitely many uniformly $F$-compatible ideals \cite{Sch09}. 
In fact, an explicit bound on the Hilbert--Samuel multiplicity of an
$F$-pure noetherian local ring \cite{HW15} allows one to obtain explicit bounds
for the number of uniformly $F$-compatible ideals of a given coheight in
the local setting 
(see also \cite{ST10} and Proposition \ref{prop:finiteness-F-compatible}). 
Thus, the finiteness of uniformly $F$-compatible ideals for a noetherian 
$F$-finite Frobenius split domain $R$ readily implies the finiteness of the set
of trace ideals of finite extensions $R$. One then shows that the 
stable trace ideal has to define the non-splinter locus. 

The drawback of the above approach is that there is typically no 
control over when the trace ideals of finite extensions
of a noetherian Frobenius split ring $R$ stabilize. Thus, 
it feels hopeless to obtain a more explicit description of 
the ideal
$\uptau_R$ that defines the closed non-splinter locus of $R$
via the approach of uniformly $F$-compatible ideals.
We devote a significant portion of our paper 
to obtaining a better understanding
of $\uptau_R$.
Our strategy involves looking at the plus
closure operation. Just as tight closure detects weak $F$-regularity, \emph{plus closure} detects
the splinter property in the sense that a noetherian domain is a splinter
precisely when all ideals of the domain are plus closed. 
By analyzing closure operations associated with $R$-algebras \cite{PRG, Hoc94}, 
we show that the ideal $\uptau_R$ that defines the non-splinter
locus of $R$ in Theorem \ref{thm:affine-main-thm} is the 
\emph{big test ideal} of plus closure. Said differently, $\uptau_R$ 
coincides with the 
ideal that one morally expects to
define the non-splinter locus of a noetherian domain.

\begin{proposition}{\normalfont(see Propositions \ref{prop:algebra-closure}, 
\ref{prop:splinter-ideal-properties} and Corollary \ref{cor:ideal-traces-approx-gor})}
\label{prop:big-test-ideal}
Let $R$ be an approximately Gorenstein noetherian domain of 
arbitrary characteristic (i.e. without any restrictions on 
characteristic) and let $\mathcal{C}$ be the collection of
finite $R$-subalgebras of $R^+$. Then we have the following:
\begin{enumerate}
    \item The ideal 
    $\uptau_R \coloneqq \bigcap_{S \in \mathcal{C}} \uptau_{S/R}$ 
    equals the big plus closure test ideal $\bigcap_{I} (I:IR^+ \cap R)$. Here the latter intersection ranges over all ideals
    of $R$.
    
    \item If $R$ is complete local and $B$ is an $R$-algebra, 
    then $\uptau_{B/R} = \bigcap_I (I:IB \cap R)$.
\end{enumerate}
\end{proposition}


\noindent The class of approximately Gorenstein rings is fairly broad
and includes 
noetherian normal rings and reduced locally excellent rings.
Taking $B = R^+$ in Proposition
\ref{prop:big-test-ideal}(2), we see that when 
$R$ is a complete local domain, the image of the
map $\Hom_R(R^+, R) \xrightarrow{\eval \MVAt 1} R$ equals
$\bigcap_I(I:IR^+ \cap R)$, which in turn equals
$\uptau_R$. This observation 
recovers a result of Hochster
and Zhang that to the best of our knowledge has not appeared in
print. 
 We refer the reader to Subsection 
\ref{subsec:Algebra closures, traces and ideal traces} 
for further details on test ideals of algebra closures,
where among other things, we partially 
answer a question raised
by P{\'e}rez and R.G. \cite{PRG} 
in the affirmative about the
equality of big and finitistic test ideals of closure
operations associated with algebras and certain modules.

Proposition \ref{prop:big-test-ideal}(1) and
ideal theoretic results from our work on permanence 
properties of splinters \cite{DT19} allow us to
obtain some transformation rules for the splinter
ideal $\uptau_R$ under
Henselizations and completions.

\begin{proposition}{\normalfont (see Proposition \ref{prop:trace-henselization-completion} and
Corollary \ref{cor:trace-commutes-completion-henselization})}
\label{prop:transformation}
Let $(R, \mathfrak m)$ be a noetherian normal domain of
arbitrary characteristic with geometrically regular 
formal fibers. Then we have the following:
\begin{enumerate}
    \item If $R^h$ is the Henselization of $R$ with
    respect to $\mathfrak{m}$, then 
    $\uptau_{R^h} \cap R =  \uptau_R$. 
    
    \item If $\widehat{R}$ is the $\mathfrak m$-adic
    completion of $R$, then $\uptau_{\widehat{R}} \cap R
    = \uptau_R$. 
\end{enumerate}
If $R$ is additionally $F$-pure, then in (1) we have $\uptau_RR^h = \uptau_{R^h}$ and in (2) we have
$\uptau_R\widehat{R} = \uptau_{\widehat R}$.
\end{proposition}

\noindent In fact, in part (1) of Proposition \ref{prop:transformation} one does
not need any assumptions on the formal fibers of $R$. One should compare
Proposition \ref{prop:trace-henselization-completion} with the 
transformation rules for the big tight closure test ideal under 
Henselizations and completions. We expect the equalities 
$\uptau_RR^h = \uptau_{R^h}$ and $\uptau_R\widehat{R} = \uptau_{\widehat R}$
to hold without restrictions on the characteristic or singularities of
normal noetherian local domains, although we are unable to show this
at present.

As an application of the openness of the splinter locus and the ascent
of the splinter property under \'etale maps \cite[Thm.\ A]{DT19}, we show
that the splinter condition for noetherian $\mathbf{N}$-graded rings over
fields is detected by the homogenous maximal ideal. This is an analogue
of Smith and Lyubeznik's result that weak $F$-regularity is detected
by the homogeneous maximal ideal \cite[Cor.\ 4.6]{SL99}.

\begin{corollary}{\normalfont (See Corollary \ref{cor:splinter-graded})}
Let $R = \bigoplus_{n = 0}^{\infty} R_n$ be a noetherian graded ring such
that $R_0 = k$ is a field.
Let $\mathfrak{m} \coloneqq \bigoplus_{n > 0} R_n$ be the homogeneous maximal
ideal of $R$. Then $R$ is a splinter if and only if $R_{\mathfrak m}$ is a
splinter.
\end{corollary}

\noindent We expect the splinter loci to be open for arbitrary 
excellent schemes in \emph{any} characteristic. At the same time, one can
use a meta construction of Hochster \cite{Hoc73(b)} to give examples of
locally excellent (but not excellent) noetherian domains whose splinter
loci are not open. The surprising aspect of Hochster's construction is 
that the local rings of these locally excellent domains are essentially
of finite type over appropriate fields. We refer the reader to 
\cite[Ex.\ 4.0.3]{DT19} for more details.

\noindent \emph{Structure of the paper:} In Section \ref{sec:prelims}
we discuss the splinter property, the notion of approximately 
Gorenstein rings and the Frobenius map. In Section \ref{sec:F-compatible} we discuss uniformly
$F$-compatible ideals and their finiteness, 
traces of algebras and a closely related notion
which we call the ideal trace, and test ideals associated to closure
operations arising from algebras. In Section \ref{sec:splinter-locus} we
first identify and prove properties of a candidate ideal that detects 
the non-splinter locus of noetherian domains in arbitrary characteristic under certain stability assumptions. 
We also discuss properties of a separable version of this ideal using
Singh's work on the separable plus closure \cite{Sin99(a)}. We finally
specialize to prime characteristic and prove our main results.

\noindent \emph{Conventions:} All rings are commutative with identity. For the most
part rings in this paper will be noetherian. However, we will 
use the absolute integral closure of a domain, which is a highly 
non-noetherian ring. We say an $R$-algebra $S$ is \emph{solid} if
there exists a nonzero $R$-linear map $S \rightarrow R$.

\section{Preliminaries}
\label{sec:prelims}

\subsection{Splinters} 
Let us introduce the main objects of investigation in this paper.

\begin{definition}
\label{def:splinter}
A noetherian ring $A$ is a \emph{splinter} if every finite ring map $A \rightarrow B$
which is surjective on $\Spec$ admits an $A$-linear left-inverse.
\end{definition}

Hochster's famous direct summand conjecture, now a theorem \cite{Hoc73, Hei02, And18, Bha18}, 
is the
assertion that noetherian regular rings are splinters. Splinters are always normal, and  
conversely, a noetherian normal $\mathbf{Q}$-algebra is always splinter. Thus, splinters 
are an interesting notion of singularity mainly in prime and mixed characteristics. 
Many naturally arising classes of rings are splinters. For example, since a direct 
summand of a regular ring is a splinter, it follows that coordinate rings of 
affine toric varieties, Veronese subrings of polynomial rings over fields and 
rings of invariants of finite groups acting on regular rings where the order of the
group is prime to the characteristic of the ring are splinters. 
Moreover, determinantal rings over fields are splinters.

A simple spreading out argument shows that the splinter condition can be checked
locally. Furthermore, because a finite direct product of noetherian rings is a splinter precisely when
the individual factors are splinters, questions about splinters often immediately reduce
to the domain case. We refer the reader to \cite{Hoc73, Hoc83, Ma88, Bha12, Ma18, DT19} 
for various properties of these
singularities. The definition of a splinter and its derived variant can also be made
in a non-Noetherian setting, and interesting non-Noetherian rings are derived splinters
\cite{AD20}.

\subsection{Approximately Gorenstein rings and purity}
Let $A$ be a ring. An $A$-linear map 
$M \rightarrow N$
is \emph{pure} (also called \emph{universally injective}) if for all
$A$-modules $P$, the induced map
\[
M \otimes_A P \rightarrow N \otimes_A P
\]
is injective. We say $M \rightarrow N$ is \emph{cyclically pure} if
for all cyclic $A$-modules $P$, $M \otimes_A P \rightarrow N \otimes_A P$
is injective, or equivalently, if for all ideals $I$ of $A$, 
$M/IM \rightarrow N/IN$ is injective.

Split maps are pure and pure maps are cyclically pure. Filtered colimits of (cyclically) pure
maps are also (cyclically) pure, and faithfully flat ring maps are pure. 
Pure ring maps induce the affine covers for the canonical Grothendieck
topology, which is the finest topology for which all representable presheaves are
sheaves \cite{AF19}.

Hochster 
characterized noetherian rings $A$ with the property that every $A$-linear map 
$A \rightarrow M$ that is cyclically pure is also pure \cite{Hoc77}. 
His characterization utilizes the following notion.

\begin{definition}
\label{def:approximately-Gorenstein}
A noetherian local ring $(A, \mathfrak m)$ is approximately Gorenstein if for 
all $n \in \mathbf{Z}_{> 0}$, there exists an ideal $I$ such that $I \subseteq \mathfrak{m}^n$
and $A/I$ is Gorenstein. A noetherian ring is
\emph{approximately Gorenstein} if all its localizations at maximal ideals are approximately Gorenstein local rings.
\end{definition}

The ideals $I$ in 
Definition \ref{def:approximately-Gorenstein} can be chosen to be $\mathfrak{m}$-primary
\cite[Prop. 2.1]{Hoc77}. The class of approximately Gorenstein rings include noetherian normal rings, 
noetherian local rings of depth at least 2, 
reduced locally excellent rings, and more generally, noetherian rings whose local rings 
at maximal ideals are formally reduced. All of this is directly proved, or 
implied by the results in \cite{Hoc77}. The point of introducing the approximately 
Gorenstein property is the following result.

\begin{proposition}\cite[Thm.\ (2.6)]{Hoc77}
\label{prop:Hochster-approx-Gorenstein}
Let $A$ be a noetherian ring. The following are equivalent:
\begin{enumerate}
    \item $A$ is approximately Gorenstein.
    \item Any ring map $A \rightarrow B$ that is cyclically pure is pure.
    \item If $M$ is an $A$-module, then any $A$-linear map $A \rightarrow M$ that is
    cyclically pure is pure.
\end{enumerate}
\end{proposition}

The approximately Gorenstein property implies that  a noetherian
domain $A$ is a splinter precisely when $A \rightarrow A^+$ is cyclically pure
\cite[Lem.\ 2.3.1]{DT19}. This observation leads to various equivalent interpretations
of an ideal that defines the non-splinter locus of a noetherian domain
in Section \ref{sec:splinter-locus}, making this ideal easier to
work with. One such interpretation involves big test ideals
of algebra closures (Proposition \ref{prop:algebra-closure}(4)), for
which we need the following result.

\begin{lemma}
\label{lem:approx-Gorenstein}
Let $A$ be an approximately Gorenstein ring.
Suppose that there exists a maximal ideal $\mathfrak m$ of $A$ and an essential extension
$A/\mathfrak{m} \hookrightarrow M$, where $M$ is a finitely generated $A$-module.
Then there exists an $\mathfrak{m}$-primary ideal $I$ of $A$ such that $A/I$ is Gorenstein
and $M$ embeds in $A/I$.
\end{lemma}

\begin{proof}
The annihilator $J$ of $M$ is $\mathfrak m$-primary. Since $A_{\mathfrak m}$
is approximately Gorenstein, there exists an $\mathfrak{m}$-primary ideal $I$ of $A$
such that $I \subseteq J$ and $A/I = A_{\mathfrak m}/IA_{\mathfrak m}$ is Gorenstein. Then
$A/\mathfrak m \hookrightarrow M$ is an essential extension of $A/I$-modules, and since 
$A/I$ is a zero-dimensional Gorenstein ring, it is an injective module over itself. 
As $\mathfrak m$ is the associated prime of $A/I$, we have an inclusion
$A/\mathfrak m \hookrightarrow A/I$, which extends to a map $f \colon M \rightarrow A/I$. 
Then $f$ is injective because $A/\mathfrak m \hookrightarrow M$ is essential.
\end{proof}

\subsection{Frobenius} Let $R$ be a ring of prime characteristic $p > 0$. Then we
have the (absolute) Frobenius map 
$$F \colon R \rightarrow R$$ 
that maps $r \mapsto r^p$. For 
$e \in \mathbf{Z}_{>0}$ we also have the $e$-th interate $F^e$ of the Frobenius. 
The target copy of $R$ regarded as an $R$-module
by restriction of scalars along $F^e$ will be denoted $F^e_*R$, and for $x \in R$
the same element viewed in $F^e_*R$ will be denoted $F^e_*x$. Thus, if $r \in R$
and $F^e_*x \in F^e_*R$, $r \cdot F^e_*x = F^e_*r^{p^e}x$.

We say $R$ is \emph{$F$-finite} if the Frobenius map is finite 
(equivalently, of finite type). We say $R$ is \emph{Frobenius split} if $F$ admits an $R$-linear left-inverse
$F_*R \rightarrow R$, and $R$ is $F$-pure if $F$ is a pure map. Frobenius 
splitting and $F$-purity do not coincide in general even for nice rings. For example, 
the first author and Murayama recently constructed examples of excellent Henselian
regular local rings of Krull dimension $1$ which are not Frobenius split \cite{DM20}. 
Regular rings of prime characteristic are always $F$-pure because the Frobenius map of
a regular ring is faithfully flat \cite{Kun69}. When $R$
is noetherian, purity of Frobenius is the same as cyclic purity \cite[Rem.\ 2.2.2]{DT19}, 
so we will not introduce a new definition for when the Frobenius is
cyclically pure even though such a definition would make sense in a non-noetherian
context.

If $R$ is a splinter of prime characteristic $p > 0$, then $R$ is $F$-pure because the
Frobenius map can be expressed as a filtered colimit of finite purely inseparable maps
which are all split by the definition of the splinter property.

Given a noetherian $R$ of prime characteristic $p > 0$, an $R$-module
$M$ and a submodule $N$ of $M$, the \emph{tight closure of $N$ in
$M$}, denoted $N^*_M$ is the set of elements $m \in M$ for which 
there exists $c \in R$ not in any minimal prime such that for all
$e \gg 0$,
\[
F^e_*c \otimes m \in \im(F^e_*R \otimes_R N \rightarrow F^e_*R \otimes_R M).
\]
For $M = R$ and $N = I$ an ideal, the tight closure of $I$ in $R$ is
usually denoted $I^*$ and consists of elements $r \in R$
for which there exists a $c$ not in any minimal prime such that
$cr^{p^e} \in I^{[p^e]}$ for all $e \gg 0$. Here $I^{[p^e]}$ is the
ideal generated by the $p^e$-th powers of elements of $I$. 
Alternatively, $F^e_*I^{[p^e]} = IF^e_*R$, the expansion of $I$ to $F^e_*R$.

\section{Uniformly \texorpdfstring{$F$}{F}-compatible and trace ideals}
\label{sec:F-compatible}

Throughout this section, $R$ will denote a ring of prime 
characteristic $p > 0$. 

\subsection{Uniformly \texorpdfstring{$F$}{F}-compatible ideals}Following Schwede 
\cite[Def.\ 3.1]{Sch10}, we make the 
following definition in a non-$F$-finite setting.

\begin{definition}
\label{def:F-compatible}
An ideal $I$ of $R$ is \emph{uniformly $F$-compatible} if for all
$e \in \mathbf{Z}_{> 0}$ and all $R$-linear maps 
$\varphi \colon F^e_*R \rightarrow R$, we have $\varphi(F^e_*I) 
\subseteq I$.
\end{definition}

\noindent Said differently, $I$ is uniformly $F$-compatible if for any 
$\varphi$ as above, we have an induced $R/I$-linear map
$\overline{\varphi} \colon F^e_*(R/I) \rightarrow R/I$
such that the diagram
\[
\begin{tikzcd}
F^e_*R \arrow{r}{\varphi} \arrow[two heads]{d}{F^e_*\pi}
               &R \arrow[two heads]{d}{\pi}\\
F^e_*(R/I) \arrow{r}{\overline{\varphi}} &R/I
\end{tikzcd}
\]
commutes, where $\pi \colon R \rightarrow R/I$ is the canonical
projection. 

Uniformly $F$-compatible ideals are related to Mehta and 
Ramanathan's notion of compatibly split subschemes of a 
Frobenius split scheme \cite{MR85}, and
indeed were named so by Schwede because of this
connection. In the noetherian local case, uniformly 
$F$-compatible ideals are dual to Smith and Lyubeznik's
$\mathcal{F}(E)$-submodules of $E$ \cite{LS01}, where 
$E$ is the injective
hull of the residue field of the noetherian local ring 
$(R, \mathfrak{m})$.

We now collect some well-known properties of uniformly 
$F$-compatible ideals. 

\begin{lemma}
\label{lem:F-compatible-properties}
Let $R$ be a ring of prime characteristic $p > 0$ and let $\Sigma$
denote the collection of uniformly $F$-compatible ideals of $R$.
Then we have the following:
\begin{enumerate}
	\item $\Sigma$ is closed under arbitrary sums and arbitrary
	intersections.
	\item If $I \in \Sigma$ and $\mathfrak p \in \Ass_R(R/I)$, 
	then $\mathfrak p \in \Sigma$.
	\item If $R$ is noetherian and $I \in \Sigma$, then $\sqrt{I} \in \Sigma$.
	\item If $R$ is noetherian, then $\nil(R) \in \Sigma$.
	\item If $R$ is Frobenius split, then every element of $\Sigma$
	is a radical ideal.
\end{enumerate}
\end{lemma}

\begin{proof}
(1) follows readily from the definition of 
uniform $F$-compatibility. For (2), assume $\mathfrak p$ is an
associated prime of $I$. Then there exists $a \notin I$ such that
$\mathfrak p = (I:a)$. Suppose $\varphi \colon F^e_*R \rightarrow R$
is an $R$-linear map. To show that $\varphi(F^e_*\mathfrak{p}) \subseteq
\mathfrak{p}$ it suffices to show that for any $r \in \mathfrak{p}$,
$\varphi(F^e_*r)a \in I$. But
\[
\varphi(F^e_*r)a = \varphi(F^e_*a^{p^e}r) \in I
\]
because $a^{p^e}r = a^{p^e-1}(ar) \in I$ and $I$ is uniformly $F$-compatible by 
assumption. Part (3) follows from (1) and (2) because
$\sqrt{I}$ is the intersection of the minimal primes of $I$,
which are associated primes of $R/I$ when $R$ is noetherian. 
Part (4) follows from (3) because the zero ideal is uniformly $F$-compatible.
Finally, (5) follows from the fact that if 
$\varphi \colon F_*R \rightarrow R$ is a Frobenius splitting of $R$,
then the induced map $\overline{\varphi} \colon F_*(R/I) 
\rightarrow R/I$ is a Frobenius splitting of $R/I$, and any 
Frobenius split ring is reduced.
\end{proof}

\begin{remark}
\label{rem:compatible-fixed-map}
Fixing an $R$-linear map $\varphi \colon F^e_*R \rightarrow R$, one
can also consider the set of ideals $I$ of $R$ that are \emph{$\varphi$-compatible} 
in the sense that $\varphi(F^e_*I) \subseteq I$.
Thus, a uniformly $F$-compatible ideal is one that is compatible
with any map $F^e_*R \rightarrow R$, for all $e > 0$. Properties 
(1)-(4) of Lemma \ref{lem:F-compatible-properties} are also satisfied
by the set of ideals compatible with a fixed $\varphi$, while 
property (5) is satisfied if $\varphi$ is a Frobenius splitting.
\end{remark}

\begin{remark}
For a fixed $R$-linear map $\varphi \colon F^e_*R \to R$, one should make a point to contrast the 
$\varphi$-compatible ideals as in Remark \ref{rem:compatible-fixed-map} with those ideals $I$ 
satisfying the stronger condition that $\varphi(F^e_*I) = I$. There is no distinction when 
$\varphi$ is surjective; when $\varphi$ is not surjective, such ideals are known to satisfy 
finiteness properties akin to those for the ideals compatible with Frobenius splittings. 
Similar remarks apply to the definition of uniformly compatible ideals in general. 
See \cite{BB11,Bli13} for further details.
\end{remark}

\subsection{Trace and ideal trace} The most important examples of uniformly 
$F$-compatible ideals in this paper
are traces of ring maps. 

\begin{definition}
\label{def:trace-ideals}
Let $A$ be a ring (not necessarily of prime characteristic) and $B$ be an 
$A$-algebra. Then the \emph{trace
 of $B$ over $A$}, denoted $\uptau_{B/A}$, is
\[
\uptau_{B/A} \coloneqq \im(\Hom_A(B,A) \xrightarrow{\textrm{eval $\MVAt 1$}} A).
\]
\end{definition}

Thus, $\uptau_{B/A} \neq 0$ precisely when $B$ admits a nonzero 
$A$-linear map $B \rightarrow A$, that is, if $B$ is a 
\emph{solid} $A$-algebra in the terminology of Hochster \cite{Hoc94}. Similarly,
$A$ is a direct summand of $B$ precisely when $\uptau_{B/A} = A$.

The next example shows that traces are  
related to the field trace from linear algebra.

\begin{example}
Suppose $A \hookrightarrow B$ is a finite extension of noetherian normal domains
which is \'etale in codimension $1$ (that is, for all height $1$ prime ideals
$\mathfrak{q}$ of $B$, $A_{\mathfrak{q} \cap A} \rightarrow B_{\mathfrak q}$ is
essentially (aka local) \'etale). 
If $K$ (resp. $L$) is the
fraction field of $A$ (resp. B), then the field trace
\[
\Tr \colon L \rightarrow K
\]
is a nonzero map because $L/K$ is separable 
\cite[\href{https://stacks.math.columbia.edu/tag/0BIL}{Tag 0BIL}]{stacks-project}. 
Since $A$ is normal,
$\Tr(B) \subseteq A$ because the minimal polynomial of any element of
$B$ over $K$ has coefficients in $A$ 
\cite[\href{https://stacks.math.columbia.edu/tag/0BIH}{Tag 0BIH}]{stacks-project}. 
Thus, restricting $\Tr$ to $B$ induces a nonzero
$A$-linear map $B \rightarrow A$, which, abusing notation, we 
also denote by $\Tr$. Then it is well-known that
$\Hom_A(B, A)$ is generated as a $B$-module by $\Tr$ 
(one can apply \cite[Prop.\ 4.8]{ST14} using the fact that the ramification
divisor of a map \'etale in codimension $1$ is trivial). That is,
any $A$-linear map $B \rightarrow A$ is of the form $\Tr(b\cdot \_)$,
for some $b \in B$. Consequently, $\uptau_{B/A}$ coincides 
with the image of
the trace map $\Tr \colon B \rightarrow A$.
\end{example}

\begin{lemma}
\label{lem:trace-ideals}
Let $A$ be a ring and $B$ be an $A$-algebra.
\begin{enumerate}
	\item $\uptau_{B/A} = \sum \im(\varphi)$, where $\varphi$
	ranges over all elements of $\Hom_A(B,A)$.
	\item If $A$ has prime characteristic $p > 0$, then
	$\uptau_{B/A}$ is uniformly $F$-compatible.
	\item If $B$ is finitely presented as an $A$-module, then
	for all prime ideals $\mathfrak{p}$ of $A$, 
	$\uptau_{B_{\mathfrak p}/A_{\mathfrak p}} = (\uptau_{B/A})_{\mathfrak p}$.
	\item If $A \rightarrow C$ is a flat ring map and $B$ is a 
	finitely presented $A$-module, then $\uptau_{B/A}C
	= \uptau_{B \otimes_A C/C}$.
\end{enumerate}
\end{lemma}

\begin{proof}
(1) The inclusion $\uptau_{B/A} \subseteq \sum \im(\varphi)$
follows by the definition of trace. Suppose $a \in \im(\varphi)$. Choose $b \in B$
such that $\varphi(b) = a$. Then $a$ is the image of $1 \in B$
under the composition
\[
B \xrightarrow{b\cdot} B \xrightarrow{\varphi} A,
\]
and so, $a \in \uptau_{B/A}$. Thus, $\im(\varphi) \subseteq
\uptau_{B/A}$ for all $\varphi \in \Hom_A(B,A)$, so we get
the other containment.

In the generality stated (i.e. without assuming $A$ is noetherian
or $B$ is module finite over $A$), (2) is proved in 
\cite[Prop.\ 8.5(i)]{DMS20}. We reproduce the proof for the
reader's convenience since the argument is straightforward and the
result is crucial for this paper. We want that if 
$\phi: F^e_*A \rightarrow A$ is an $A$-linear map, then
$\phi(F^e_*\uptau_{B/A}) \subseteq \uptau_{B/A}$. So let 
$F^e_*b \in F^e_*\uptau_{B/A}$. Choose an $A$-linear map
\[
\eta_b: B \rightarrow A
\]
such that $\eta_b(1) = b$. Then the composition
\[
B \xrightarrow{F^e} F^e_*B \xrightarrow{F^e_*\eta_b}F^e_*A \xrightarrow{\phi} A
\]
is an $A$-linear map that sends $1 \mapsto \phi(F^e_*b)$. Thus, 
$\phi(F^e_*b) \in \uptau_{B/A}$,
and so, $\phi(F^e_*\uptau_{B/A}) \subseteq \uptau_{B/A}$.


For (3), the hypothesis that $B$ is finitely presented as
an $A$ module implies that 
$$A_\mathfrak{p} \otimes_A \Hom_A(B,A) = 
\Hom_{A_\mathfrak{p}}(B_\mathfrak{p}, A_\mathfrak{p}).$$ 
Applying $A_{\mathfrak p} \otimes_A -$ to $\Hom_A(B,A) \xrightarrow{
\textrm{eval $\MVAt 1$}} A$ gives
\[
\Hom_{A_{\mathfrak p}}(B_{\mathfrak p}, A_{\mathfrak p}) \xrightarrow{
\textrm{eval $\MVAt 1$}} A_{\mathfrak p}.
\]
Since localization commutes with taking images of linear maps,
we get the desired result.

(4) is a generalization of (3). We have an exact sequence
\[
\Hom_A(B, A) \xrightarrow{\eval \MVAt 1} A \rightarrow A/\uptau_{B/A} \rightarrow 0.
\]
Applying $C \otimes_A -$ to the above sequence and using the fact that $\Hom$ commutes
with flat base change when the first argument is finitely presented, we get 
the exact sequence
\[
\Hom_C(B \otimes_A C, C) \xrightarrow{\eval \MVAt 1} C \rightarrow C/\uptau_{B/A}C \rightarrow 0.
\]
Then $\uptau_{B \otimes_A C/C} = \im(\Hom_C(B \otimes_A C, C) \xrightarrow{\eval \MVAt 1} C)
 = \ker(C \rightarrow C/\uptau_{B/A}C) = \uptau_{B/A}C$.
\end{proof}

We now deduce some non-obvious consequences of the 
previous results.

\begin{corollary}
\label{cor:trace-F-split}
Let $R$ be a ring of prime characteristic $p > 0$ and $S$ be an
$R$-algebra.
\begin{enumerate}
	\item If $R$ is Frobenius split, then $\uptau_{S/R}$ is a 
	radical ideal.
	\item If $(R, \mathfrak m)$ is a complete local noetherian
	domain and $S$ is a big Cohen--Macaulay $R$-algebra (for 
	example, $S = R^+$), then
	$\uptau_{S/R}$ is a nonzero uniformly $F$-compatible ideal.
\end{enumerate}
\end{corollary}

\begin{proof}
(1) follows by Lemma \ref{lem:trace-ideals}(2) and Lemma 
\ref{lem:F-compatible-properties}(3).

For (2), by the defining property of a big Cohen--Macaulay
$R$-algebra, it follows that if $d = \dim(R)$, then
\[
H^d_{\mathfrak m}(S) \neq 0.
\]
Consequently, $\Hom_R(S,R) \neq 0$ by \cite[Cor.\ 2.4]{Hoc94}, or equivalently,
$\uptau_{S/R}$ is nonzero.
\end{proof}

\begin{example}
There are other important examples of uniformly $F$-compatible
ideals. 
\begin{enumerate}
\item Suppose $R$ is a reduced $F$-finite noetherian ring. The 
\emph{big} or \emph{non-finitistic} test ideal $\uptau_b(R)$ of $R$ is a
uniformly $F$-compatible ideal, which is also the smallest 
(with respect to inclusion) 
uniformly 
$F$-compatible ideal of $R$ that is not contained
in any minimal prime \cite{Vas98, HT04, Sch10}. 
Blickle, Schwede and Tucker
have shown that when $R$ is additionally a normal, $F$-finite and 
$\mathbf{Q}$-Gorenstein domain, 
then $\uptau_b(R)$ can be realized as the
trace  of some finite and generically \'etale extension $S$ of
$R$ \cite[Thm.\ 4.6]{BST15}. In a similar vein, Polstra and 
Schwede have proved that in the $\mathbf{Q}$-Gorenstein
setting, one can often realize any uniformly 
$F$-compatible ideal of $R$ as a trace of
some finite extension of $R$ \cite{PS20}. 
However, 
it is not known if $\uptau_b(R)$ can be realized as trace
of a solid $R$-algebra $S$ when $R$ is not $\mathbf{Q}$-Gorenstein.
We will show in Corollary \ref{cor:test-ideal-trace} 
that the finitistic test ideal is recoverable
as a trace of a big Cohen-Macaulay algebra for 
an arbitrary complete local
domain of prime characteristic.

{\vspace{2mm}}

\item 
Suppose $\mathcal{C}$ is a 
collection of ideals of a noetherian ring $R$ of prime characteristic $p> 0$
that is closed under Frobenius powers, that is, $I \in \mathcal{C}
\Rightarrow I^{[p^e]} \in \mathcal{C}$ for all 
$e \in \mathbf{Z}_{> 0}$. 
We claim that
\[
\mathcal{I} \coloneqq \bigcap_{I \in \mathcal{C}} (I:I^*)
\]
is uniformly $F$-compatible. Here $I^*$ denotes the tight closure
of $I$. Suppose $c \in \mathcal{I}$ and 
$\varphi \colon F^e_*R \rightarrow R$ is an $R$-linear map. We
have to show that $\varphi(F^e_*c) \in \mathcal{I}$, that is, for 
$I \in \mathcal{C}$ and $z \in I^*$, we want $\varphi(F^e_*c)z \in 
I$. Now $z \in I^*$ implies $z^{p^e} \in (I^{[p^e]})^*$. 
This follows by choosing $d \in R^{\circ}$ such that 
$d(z^{p^e})^{p^f} = dz^{p^{e+f}} \in I^{[p^{e+f}]}
= (I^{[p^e]})^{[p^f]}$ for $f \gg 0$. Now $I \in \mathcal{C}
\Rightarrow I^{[p^e]} \in \mathcal{C}$. So by the choice of
$c$, we have $cz^{p^e} \in c(I^{[p^e]})^* \subseteq I^{[p^e]}$. 
Then
$\varphi(F^e_*c)z = \varphi(F^e_*cz^{p^e}) \in \varphi(F^e_*I^{[p^e]})
\subseteq I$, as desired. If $R$ is an approximately Gorenstein
ring (for example, an excellent reduced ring or a normal ring),
then taking $\mathcal C$ to be the collection of all ideals of $R$, 
the ideal $\mathcal{I}$ is the \emph{finitistic test ideal} of 
$R$ \cite[Prop.\ (8.15)]{HH90}. 
Similarly, taking $\mathcal{C}$ to be the 
collection of parameter ideals of $R$, the ideal $\mathcal I$ 
is the 
\emph{parameter test ideal} of $R$ \cite[Def.\ 4.3]{Smi95}. 
Thus, both types of test 
ideals are uniformly $F$-compatible for nice rings. 
\end{enumerate}
\end{example}

The trace $\uptau_{B/A}$ detects whether $A \rightarrow B$
splits. However, splitting is not a good notion for maps
$A \rightarrow B$ without finiteness assumptions on $B$. 
The better notion then is that of purity, and 
we now introduce a uniformly $F$-compatible that
detects purity of $A \rightarrow B$ in most cases of interest. 

\begin{definition}
\label{def:ideal-trace}
Let $A$ be a ring and $B$ be an $A$-algebra. Then the \emph{ideal trace} of $B/A$,
denoted $T_{B/A}$, is 
\[
T_{B/A} \coloneqq \bigcap_{I} (I: IB\cap A),
\]
where the intersection ranges over all ideal $I$ of $A$.
\end{definition}

Ideal traces satisfy the following elementary properties.

\begin{lemma}
\label{lem:contraction-ideals}
Let $A$ be a ring and $B$ be an $A$-algebra. 
Suppose $\mathcal C$
is the set of ideals of $A$. 
\begin{enumerate}
	\item $T_{B/A} = A$ if and only if $A \rightarrow B$ is
	cyclically pure.
	\item If $A$ is an approximately Gorenstein ring,
	then $T_{B/A} = A$ if and only if $A \rightarrow B$ is pure.
	\item We have $\uptau_{B/A} \subseteq T_{B/A}$. 
	\item Suppose $A$ has prime characteristic $p > 0$. Then for 
	a fixed ideal $I$ of $A$,
	\[
	T_I \coloneqq \bigcap_{e \geq 0} (I^{[p^e]}: I^{[p^e]}B\cap A)
	\]
	is uniformly $F$-compatible.
	\item Suppose $A$ has prime characteristic $p > 0$. Then
	$T_{B/A}$ is uniformly $F$-compatible.
\end{enumerate}
\end{lemma}

\begin{proof}
(1) follows from the fact that for an ideal $I$ of $A$,
$A/I \rightarrow B/IB$ is injective if and only if 
$I = IB \cap A$, or equivalently, that $(I: IB \cap A) = A$.

(2) follows from (1) and Proposition \ref{prop:Hochster-approx-Gorenstein} due to Hochster.

(3) Let $c \in \uptau_{B/A}$, and choose $\varphi \in \Hom_A(B,A)$
such that $\varphi(1) = c$. Then for any ideal $I$ of $A$, 
\[
c(IB\cap A) = \varphi(1)(IB \cap A) = \varphi(IB \cap A) 
\subseteq \varphi(IB) \subseteq I,
\] 
where the second equality and the last
containment of sets follow
by $A$-linearity of $\varphi$.

(4) Suppose $c \in T_I$ and
$\varphi: F^e_*A \rightarrow A$ is an $A$-linear map. For 
$f \geq 0$, let 
$z \in I^{[p^f]}B \cap A$. Then $z^{p^e} \in I^{[p^{e+f}]}B\cap A$,
and so, 
\[
\varphi(F^e_*c)z = \varphi(F^e_*cz^{p^e}) \in 
\varphi(F^e_*I^{[p^{e+f}]}) \subseteq I^{[p^f]}.
\]
Thus, $\varphi(F^e_*c)(I^{[p^f]}B\cap A) \subseteq I^{[p^f]}$ for
all $f \geq 0$, which shows $\varphi(F^e_*c) \in T_I$, 
and hence, the uniform $F$-compatibility of $T_I$.

(5) It is clear that 
\[
T_{B/A} = \bigcap_{I \in \mathcal{C}} T_I,
\]
and since an arbitrary intersection of uniformly $F$-compatible 
ideals is uniformly $F$-compatible
(Lemma \ref{lem:F-compatible-properties}), 
by (4) we conclude that
$T_{B/A}$ is uniformly $F$-compatible.
\end{proof}

\begin{remark}
	In general, the containment 
	$\uptau_{B/A} \subseteq T_{B/A}$ is strict even for nice rings
	$A$. For example, the first author and Murayama have recently
	constructed examples of excellent Henselian regular local rings
	$A$ of Krull dimension $1$ and prime characteristic $p > 0$ 
	that admit no nonzero $A$-linear maps $F_*A \rightarrow A$ \cite{DM20}. 
	For such a ring, $\uptau_{F_*A/A} = 0$ . 
	However, the Frobenius
	$F: A \rightarrow F_*A$ is faithfully flat \cite{Kun69}, 
	hence is pure,
	and hence is also cyclically pure. Therefore $T_{F_*A/A} = A$.
	This example is extreme in the sense that $F_*A$ is not a solid $A$-algebra. Thus one can ask
	the following question: suppose $B$ is a solid $A$-algebra
	and $T_{B/A} = A$. Then does it follow that $\uptau_{B/A} = A$
	when $A$ is approximately Gorenstein?
	The question has an affirmative answer if 
	$A \rightarrow B$ is 
	finite or if $A$ is complete,
	because then $A \rightarrow B$ is pure, and hence split,
	by \cite[Cor.\ 5.2]{HR76} in the finite case and by a lemma
	due to Auslander in the complete case (see proof of Corollary
	\ref{cor:complete-local}). 

\end{remark}

\subsection{Algebra closures, traces and ideal traces}
\label{subsec:Algebra closures, traces and ideal traces}
Let $A$ be a noetherian
ring of arbitrary characteristic. Then for any $A$-algebra $B$,
P\'erez and R.G. define an associated closure operation 
$\cl_B$ that satisfies many of the properties of tight closure \cite[Def.\ 2.4]{PRG}. Namely,
for an arbitrary $A$-module $M$ and a submodule $N$ of $M$,
an element $m \in M$ is in $N^{\cl_B}_M$, the
\emph{$\cl_B$ closure of $N$ in $M$}, if 
\[
1 \otimes m \in \im(B \otimes_A N \rightarrow B \otimes_A M),
\]
where $B \otimes_A N \rightarrow B \otimes_A M$ is obtained by tensoring $N \subseteq M$
by $\id_B$.

For example, if $M = A$ and $N = I$ is an ideal, then 
\[
I^{\cl_B}_A = IB \cap A.
\]
If $B = A^+$, then $\cl_{A^+}$ is commonly known as the
\emph{plus closure}.

Analogous to tight closure, 
one defines the \emph{big test ideal} of $\cl_B$, denoted $\uptau_{\cl_B}(A)$, as
\[
\uptau_{\cl_B}(A) = \bigcap_{N \subseteq M} (N:_A N^{\cl_B}_M),
\]
where the intersection ranges over all $A$-modules $M$ and submodules $N$ of $M$.
Similarly, the \emph{finitistic test ideal} of $\cl_B$, denoted $\uptau^{\fg}_{\cl_B}(A)$,
is defined as
\[
\uptau^{\fg}_{\cl_B}(A) = \bigcap_{\textrm{$M$\ is\ fin.\ gen.}} (N:_A N^{\cl_B}_M).
\]
Now the intersection runs over all finitely generated $R$-modules $M$ and submodules $N$.

Despite the many parallels between tight closure and $\cl_B$, 
these latter closure operations are better behaved. In particular, our next
result answers \cite[Question 3.7]{PRG} in the affirmative for 
closure operations that arise from $A$-algebras (see also
Remark \ref{rem:closure-modules}(2) for a partial result
for closure operations arising from $A$-modules).

\begin{proposition}
\label{prop:algebra-closure}
Let $A$ be a noetherian ring and $B$ be an $A$-algebra. Let $M$
be an $A$-module and $N$ be a submodule of $M$. Then we have the following:
\begin{enumerate}
    \item If $M'$ is a submodule of $M$, then 
    $(N \cap M')^{\cl_B}_{M'} \subseteq N^{\cl_B}_M$.
    
    \item Let $\{M_i\}_i$ be the collection of finitely generated $A$-submodules
    of $M$. Then $\{(N \cap M_i)^{\cl_B}_{M_i}\}_i$ is a filtered poset of 
    submodules of $M$ under inclusion and
    \[
    N^{\cl_B}_M = \bigcup_i (N \cap M_i)^{\cl_B}_{M_i}.
    \]
    
    \item $\uptau_{\cl_B}(A) = \uptau^{\fg}_{\cl_B}(A)$, that is, the big and finitistic
    test ideals of $\cl_B$ coincide.
    
    \item If $A$ is approximately Gorenstein, 
    then $\uptau_{\cl_B}(A) = T_{B/A}$.
    
    \item If $\mathcal C$ is the collection of ideals of $A$ primary 
    to maximal ideals, then $T_{B/A} = \bigcap_{I \in \mathcal C}(I: IB \cap A)$.
    
    \item $\uptau_{B/A} \subseteq \uptau_{\cl_B}(A)$.
\end{enumerate}
\end{proposition}

\begin{proof}
For ease of notation, for any $A$-module $M$ and a submodule $N$ of $M$,
we use $\xi_{N,M}$ to denote the canonical map 
$B \otimes_A N \rightarrow B \otimes_A M$ obtained by tensoring $N \subseteq M$
by $\id_B$. By definition, $m \in N^{\cl_B}_M$ if and only if 
$1 \otimes m \in \im(\xi_{N,M})$.

(1) We have a commutative diagram
\[
\begin{tikzcd}
B \otimes_A (N \cap M') \arrow{rr}{\xi_{N \cap M', M'}} \arrow{d}{\xi_{N \cap M', N}}
               && B \otimes_A M' \arrow{d}{\xi_{M', M}}\\
B \otimes_A N \arrow{rr}{\xi_{N,M}} && B \otimes_A M
\end{tikzcd}
\]
If $m' \in (N \cap M')^{\cl_B}_{M'}$, then $1 \otimes m' \in \im(\xi_{N \cap M', M'})$.
By the commutativity of the above diagram, it follows that $1 \otimes m' \in \im(\xi_{N,M})$.
Thus, $(N \cap M')^{\cl_B}_{M'} \subseteq N^{\cl_B}_M$.

(2) $\{M_i\}_i$ is a filtered poset of $R$-submodules of $M$ under inclusion such that
\[
\textrm{$M = \colim_i M_i$ \hspace{3mm} and \hspace{3mm} $N = \colim_i N \cap M_i$.}
\]
The colimits here are unions. By (1), we get
\[
\bigcup_i (N \cap M_i)^{\cl_B}_{M_i} \subseteq N^{\cl_B}_M.
\]
Since tensor products commute with filtered colimits, it
follows that
\[
\xi_{N, M} \colon B \otimes_A N \rightarrow B \otimes_A M = 
\colim_i\big{(}\xi_{N \cap M_i, M_i} \colon B \otimes_A (N \cap M_i) \rightarrow B \otimes_A M_i\big{)}.
\]
Consequently, by the exactness of filtered colimits 
in $\Mod_A$ \cite[\href{https://stacks.math.columbia.edu/tag/00DB}{Tag 00DB}]{stacks-project},
it follows that
\[
\im(\xi_{N,M}) = \colim_i \im(\xi_{N \cap M_i, M_i}).
\]
Thus, for $m \in M$, in order for $1 \otimes m$ to be in $\im(\xi_{N,M})$,
there must exist an index $i$ such that $m \in M_i$ and
$1 \otimes m \in \im(\xi_{N \cap M_i, M_i})$. Unravelling the definition
of $\cl_B$, this gives (2).

(3) The containment $\uptau_{\cl_B}(A) \subseteq \uptau_{\cl_B}^{\fg}(A)$
follows by the definitions of the big and finitistic test ideals of $\cl_B$.
Let $a \in \uptau_{\cl_B}^{\fg}(A)$, and $N \subseteq M$ be a pair of
$A$-modules. Given $m \in N^{\cl_B}_M$, by (2) there exists some
finitely generated submodule $M'$ of $M$ such that 
$m \in (N \cap M')^{\cl_B}_{M'}$. Since 
\[
a\big{(}(N \cap M')^{\cl_B}_{M'}\big{)} \subseteq N \cap M',
\]
it follows that $am \in N \cap M' \subseteq N$.
Thus,
\[
a(N^{\cl_B}_M) \subseteq N,
\]
and since $M$ and $N$ are arbitrary, we get $a \in \uptau_{\cl_B}(A)$. This shows that
$\uptau^{\fg}_{\cl_B}(A) \subseteq \uptau_{\cl_B}(A)$.

(4) Recall that $$T_{B/A} = \bigcap_I (I: IB \cap A) = \bigcap_I (I: I^{\cl_B}_A),$$ 
where $I$ ranges over all ideals
of $A$. By (3) it suffices to show that $\uptau^{\fg}_{\cl_B}(A) = T_{B/A}$. We will 
follow the proof of \cite[Prop.\ (8.15)]{HH90}, where the analogous fact is shown for
the finitistic tight closure test ideal. Let $c \in T_{B/A}$. By \cite[Lem.\ 3.3]{PRG},
it suffices to show that if $M$ is a finitely generated $A$-module, then 
\[
c \in (0:_A 0^{\cl_B}_M),
\]
that is, $c$ annihilates $0^{\cl_B}_M$. Assume for contradiction that 
this is not the case. Then there 
exists $m \in 0^{\cl_B}_M$ such that $cm \neq 0$. Since $M$ is a noetherian $A$-module,
let $N$ be a submodule of $M$ maximal with respect to the property that $cm \notin N$.
Then $m \in N^{\cl_B}_M$ and $cm \notin N$. Replacing $M$ by $M/N$, $N$ by $0$
and $m$ by its image in $M/N$, we may assume there exists 
$m \in 0^{\cl_B}_M$ such that $cm \neq 0$ and 
for all submodules $0 \subsetneq M' \subsetneq M$,
$cm \in M'$. Thus, $A(cm) \subset M$ is an essential extension, and moreover,
$A(cm)$ has to be a nontrivial simple $A$-module. 
This means that there exists a maximal ideal $\mathfrak{m}$
of $A$ such that $A(cm) \simeq A/\mathfrak{m}$. Since $A$
is approximately Gorenstein, $M$ embeds in $A/I$ for an $\mathfrak{m}$-primary
ideal $I$ by Lemma \ref{lem:approx-Gorenstein}. Then
\[
0^{\cl_B}_M \subseteq 0^{\cl_B}_{A/I}.
\]
Since $c \in T_{B/A} = (I: IB \cap A) = (I: I^{\cl_B}_A)$, \cite[Lem.\ 2.15]{PRG}
shows that $c$
annihilates $0^{\cl_B}_{A/I}$, and so,
$c$ also annihilates $0^{\cl_B}_M$. This is a contradiction.

(5) Recall $\mathcal{C}$ is the collection of ideals of $A$ primary to maximal ideals.
It suffices to show that 
\[
\bigcap_{I \in \mathcal C} (I:IB\cap A) \subseteq T_{B/A}.
\]
Let $J$ be an arbitrary ideal of $A$ and let 
$c \in \bigcap_{I \in \mathcal C} (I:IB\cap A)$. Then for any maximal ideal 
$\mathfrak m$ of $A$ and $n \in \mathbf{Z}_{> 0}$,
\[
c(JB \cap A) \subseteq c((J + \mathfrak{m}^n)B \cap A) \subseteq J + \mathfrak{m}^n.
\]
Thus,
\[
c(JB \cap A) \subseteq \bigcap_{n \in \mathbf{Z}_{> 0}}  
\bigcap_{\mathfrak m} J + \mathfrak{m}^n  = J,
\]
where the inner intersection runs over all maximal ideals of $A$. It follows
that $c \in T_{B/A}$.

(6) If $A$ is approximately Gorenstein, then (6) follows from
(4) and Lemma \ref{lem:contraction-ideals}(3) because 
$\uptau_{B/A} \subseteq T_{B/A}$. We will show that (6) holds for any noetherian ring $A$. Let $c \in \uptau_{B/A}$ and choose an $A$-linear map
$f \colon B \rightarrow A$
such that $f(1) = c$. Then for any $A$-module $N$, we get an 
$A$-linear map
\[
  \begin{tikzcd}[column sep=1.475em,row sep=0]
   \eta_N \colon B \otimes_A N \rar & N.\\
    b \otimes n \rar[mapsto] & f(b)n
  \end{tikzcd}
\]
Note that $\eta_N$ is
functorial in $N$, that is, for any $A$-linear map 
$\varphi \colon N \rightarrow M$, 
\[
\begin{tikzcd}
B \otimes_A N \arrow{rr}{\id_B \otimes \varphi} \arrow{d}{\eta_N}
               && B \otimes_A M \arrow{d}{\eta_M}\\
N \arrow{rr}{\varphi} && M.
\end{tikzcd}
\]
commutes. Applying this commutative diagram when $N$ is a submodule of an
$A$-module $M$ and $\varphi$ is the inclusion map, we see that
if $m \in N^{\cl_B}_M$, that is, if  
$1 \otimes m \in \im(B \otimes_A N \rightarrow B \otimes_A M)$,
then
\[
cm = f(1)m = \eta_M(1 \otimes m) 
\]
is in the image of $N \subseteq M$, that is, $cm \in N$. Thus,
$c \in (N:_A N^{\cl_B}_M)$, and since $N, M$ are arbitrary, we get
$c \in \uptau_{\cl_B}(A)$.
\end{proof}

\begin{remark}
\label{rem:closure-modules}
{\*}
\begin{enumerate}
\item Proposition \ref{prop:algebra-closure}(6) is also observed when
$A$ is local in \cite[Thm.\ 3.12]{PRG}
(see \cite[Rem.\ 3.13]{PRG}), where it is additionally
shown that equality holds in part (6) if $B$ is a finite $A$-algebra
or if $A$ is complete and $B$ is an arbitrary $A$-algebra. 
We will use \cite[Thm.\ 3.12]{PRG}
to show that the trace $\uptau_{B/A}$ and the ideal trace $T_{B/A}$
are equal for module finite extensions in most cases of interest
(Corollary \ref{cor:ideal-traces-approx-gor}(1)). This
basic observation will be useful in identifying and proving
properties of an ideal that detects
the non-splinter locus of a noetherian domain in Section \ref{sec:splinter-locus}.

\vspace{1mm}

\item One can define a closure operation $\cl_B$ for an $A$-module
$B$ that is not necessarily an $A$-algebra \cite[Def.\ 2.4]{PRG}. 
The proof of Proposition 2.3.1(3) can be modified to show that if $B$
is a finitely generated $A$-module, then the big and finitistic 
$\cl_B$-closure test ideals coincide. Indeed, if $B$
is generated as an $A$-module by $b_1, \dots, b_n$, then
$m \in N^{\cl_B}_M$ precisely when $b_i \otimes m 
\in \im (B \otimes_A N \to B \otimes_A M)$ for all $i$. But one
can always find a large finitely generated $A$-submodule $M'$
of $M$ such that $m \in M'$ and 
$b_i \otimes m \in \im (B \otimes_A (N \cap M') 
\to B \otimes_A M')$ for all $i$.
However, the closure operations associated with modules
that are not algebras are often not related to tight closure.
For example, there exist even
finitely generated Cohen--Macaulay modules $B$ over a non-regular
but weakly $F$-regular local ring for
which $\cl_B$ is non-trivial \cite[Rem.\ 2.23]{PRG}.


\vspace{1mm}

\item The arguments of Proposition 
\ref{prop:algebra-closure} will not work to show that the big and finitistic
test ideals from tight closure theory are equal because tight closure involves
checking algebra closure type relations for infinitely many $B$'s. However,
see Corollary \ref{cor:test-ideal-trace} for a deep characterization
of finitistic tight closure in terms of an algebra closure due to Hochster.
\end{enumerate}
\end{remark}

We can now exhibit relations between the ideals $\uptau_{B/A}$
and $T_{B/A}$ for nice rings. In particular, we will see that the
ideal trace $T_{B/A}$ often localizes for finite extensions $A \hookrightarrow B$,
a fact that is not obvious from its definition which involves an
infinite intersection of ideals.

\begin{corollary}
\label{cor:ideal-traces-approx-gor}
Let $A$ be an approximately Gorenstein noetherian ring and $B$ be an $A$-algebra. 
\begin{enumerate}
    \item If $B$ is a finite $A$-algebra, then $\uptau_{B/A} = \uptau_{\cl_B}(A) =
    \uptau^{\fg}_{\cl_B}(A) = T_{B/A}$.
    \item If $B$ is a finite $A$-algebra, then for all prime ideals $\mathfrak p$
    of $A$, $T_{B_{\mathfrak p}/A_{\mathfrak p}} = (T_{B/A})_{\mathfrak p}$.
    \item If $A$ is complete local, then $\uptau_{B/A} = \uptau_{\cl_B}(A) =
    \uptau^{\fg}_{\cl_B}(A) =  T_{B/A}$.
    \item If $A$ is a complete local domain,
    then $\uptau_{A^+/A} = \uptau_{\cl_{A^+}}(A) =
    \uptau^{\fg}_{\cl_{A^+}}(A) =  T_{A^+/A}$ and $\uptau_{A^+/A} \neq 0$.
\end{enumerate}
All the ideals in (1) and (3)
are nonzero precisely when $B$ is a solid $A$-algebra.
\end{corollary}

\begin{proof}
(1) The equalities $\uptau_{\cl_B}(A) = \uptau^{\fg}_{\cl_B}(A) = T_{B/A}$ follow by
Proposition \ref{prop:algebra-closure}. Thus, it suffices to show that
\begin{equation}
\label{eq:trace-ideal-trace}
\uptau_{B/A} = T_{B/A}.
\end{equation}
This equality holds when $A$ is local and $B$ is a finite $A$-algebra 
by \cite[Thm.\ 3.12]{PRG}, where it is
shown that $\uptau_{B/A} = \uptau_{\cl_B}(A)$. 
Now note that
\begin{equation}
\label{eq:localization}
(T_{B/A})_{\mathfrak p} = \bigg{(} \bigcap_I (I: IB \cap A) \bigg{)}_{\mathfrak p} 
\subseteq \bigcap_I (I: IB \cap A)_{\mathfrak p} = 
\bigcap_{I} (IA_{\mathfrak p}: IB_{\mathfrak p} \cap A_{\mathfrak p}) = T_{B_\mathfrak{p}/A_{\mathfrak{p}}}.
\end{equation}
The first equality is by definition of $T_{B/A}$, and the inclusion is a well-known
property of localization. The second equality follows by flatness of localization.
Indeed, we have $(I:IB \cap A)_{\mathfrak p} = (IA_{\mathfrak p}: (IB \cap A)A_{\mathfrak p})$
by flatness and the fact that $IB \cap A$ is 
finitely generated, and 
\[
A/(IB \cap A) \hookrightarrow B/IB
\]
stays an inclusion upon localizing at $\mathfrak p$, which implies that $(IB \cap A)A_{\mathfrak p}
= IB_{\mathfrak p} \cap A_{\mathfrak p}$. The final equality follows by definition of
$T_{B_\mathfrak{p}/A_\mathfrak{p}}$ and the fact that all ideals of $A_{\mathfrak p}$ are expanded from $A$.

By the veracity of (\ref{eq:trace-ideal-trace}) in the local case, we have 
$T_{B_\mathfrak{p}/A_\mathfrak{p}}
= \uptau_{B_\mathfrak{p}/A_\mathfrak{p}}$, and consequently,
\[
(\uptau_{B/A})_{\mathfrak p} \subseteq (T_{B/A})_{\mathfrak p} \subseteq T_{B_\mathfrak{p}/A_\mathfrak{p}}
= \uptau_{B_\mathfrak{p}/A_\mathfrak{p}} = (\uptau_{B/A})_{\mathfrak p}.
\]
Here the first inclusion follows because $\uptau_{B/A} \subseteq T_{B/A}$ 
by Lemma \ref{lem:contraction-ideals}(3), the second inclusion
follows by (\ref{eq:localization}) and the last equality follows 
by Lemma \ref{lem:trace-ideals}(3). Thus,
for all prime ideals $\mathfrak p$ of $A$, $(\uptau_{B/A})_{\mathfrak p} = (T_{B/A})_{\mathfrak p}$, which 
implies that $\uptau_{B/A} = T_{B/A}$.

(2) follows from (1) and the fact that $(\uptau_{B/A})_{\mathfrak p} = \uptau_{B_\mathfrak{p}/A_\mathfrak{p}}$ 
by Lemma \ref{lem:trace-ideals}(3).

(3) The equalities $\uptau_{\cl_B}(A) =
    \uptau^{\fg}_{\cl_B}(A) =  T_{B/A}$
again follow by Proposition \ref{prop:algebra-closure}, and 
\cite[Thm.\ 3.12]{PRG} shows that $\uptau_{B/A} = \uptau_{\cl_B}(A)$.

The first part of (4) follows by (3). It remains to show
that $A^+$ is a solid $A$-algebra.
Since $A$ is complete local, by Cohen's structure theorem, there exists
a complete regular local ring R and a module finite extension 
$R \hookrightarrow A$. Then $A^+ = R^+$ and the map 
$R \hookrightarrow A \hookrightarrow A^+$ is pure because $R$ is a splinter
by the direct summand theorem \cite{And18, Hoc73}. Since $R$ is complete,
a well-known result due to Auslander (see for example 
\cite[Lem.\ 2.3.3]{DMS20}) 
 implies $R \rightarrow A^+$ splits. Since $A$ is a finite extension
domain of $R$, $A^+$
must also be a solid $A$-algebra by \cite[Cor.\ 2.3]{Hoc94}.
\end{proof}

\begin{remark}
Corollary \ref{cor:ideal-traces-approx-gor}(4) was announced a number of years ago by
Hochster and Zhang. To the best of our knowledge, their result is
not publicly available and our proof is independent of theirs.
\end{remark}

Granting another unpublished but widely available result of Hochster in his
course notes \cite{Hoc07}, 
we show that the finitistic test ideal from
tight closure theory can always be recovered as a trace ideal for \emph{any}
complete local domain. This
result is an analog of
\cite[Thm.\ 4.6]{BST15}.

\begin{corollary}
\label{cor:test-ideal-trace}
Let $(R, \mathfrak m)$ be a complete local noetherian domain of prime
characteristic $p > 0$. Then there exists a big Cohen-Macaulay $R$-algebra
$B$ that satisfies all of the following properties:
\begin{enumerate}
    \item $B$ is an $R^+$-algebra.
    \item For all finitely generated $R$-modules $M$ and submodules $N$ of
    $M$, $N^{\cl_B}_M = N^*_M$, where $N^*_M$ denotes the tight closure of
    $N$ in $M$.
    \item $\uptau_{B/R} = \uptau^{\fg}_*(R)$, where $\uptau^{\fg}_*(R)$ is the finitistic
    tight closure test ideal of $R$.
\end{enumerate}
\end{corollary}

\begin{proof}
The existence of a big Cohen--Macaulay $R$-algebra $B$ that satisfies
(1) and (2) follows by \cite[Thm.\ on\ pg.\ 250]{Hoc07}. This is a 
significant strengthening of \cite[Thm.\ (11.1)]{Hoc94}, where
Hochster shows that given a 
finitely generated $R$-module $M$ and a submodule $N$ of $M$, there exists
a big 
Cohen--Macaulay $R$-algebra $B$ depending on $M, N$ 
such that $N^{\cl_B}_M = N^*_M$.

So choose $B$ satisfying (1) and (2). Then $B$ is a solid $R$-algebra because
$R$ is complete and $B$ is big Cohen--Macaulay. Now by \cite[Thm.\ 3.12]{PRG},
\[
\uptau_{B/R} = \uptau_{\cl_B}(R).
\]
But Proposition \ref{prop:algebra-closure}(3) tells us that
\[
\uptau_{\cl_B}(R) = \uptau^{\fg}_{\cl_B}(R) \coloneqq 
\bigcap_{\textrm{$M$\ is\ fin.\ gen.}} (N:_R N^{\cl_B}_M) = 
\bigcap_{\textrm{$M$\ is\ fin.\ gen.}} (N:_R N^{*}_M).
\]
The last intersection is precisely the finitistic tight closure 
test ideal of $R$  \cite[Def.\ 8.22]{HH90}.
\end{proof}

\begin{remark}
Let $B$ be Hochster's (very large) big Cohen--Macaulay $R$-algebra from
Corollary \ref{cor:test-ideal-trace}. It is not known if $\cl_B$ coincides
with tight closure for submodules of arbitrary $R$-modules. Indeed, an
affirmative answer to this question would imply that $\uptau_{B/R}$ is also
the non-finitistic or big test ideal $\uptau_b(R)$ of $R$ from tight closure,
thereby showing that the big and finitistic tight closure test ideals coincide.
This would then settle the outstanding problem of whether weak $F$-regularity
is equivalent to strong $F$-regularity for complete local domains.
\end{remark}

\subsection{Finiteness of uniformly \texorpdfstring{$F$}{F}-compatible ideals} Let $R$ be a 
noetherian ring, and let $\Sigma$ be a collection of radical ideals of 
$R$ that is closed under finite intersections and such that for
$I \in \Sigma$, any minimal prime of $I$ is also in $\Sigma$.
The main example for us is where $\Sigma$ is the collection of
uniformly $F$-compatible ideals of $R$ when $R$ is Frobenius 
split; see Lemma \ref{lem:F-compatible-properties}. It follows that 
$\Sigma$ is finite precisely when 
$\Sigma \cap \Spec(R)$ is finite because every ideal of $\Sigma$
(apart from the unit ideal) is a finite intersection of
elements of $\Sigma \cap \Spec(R)$. We use this basic observation
to recall the following well-known result which is a key
technical ingredient of our paper.

\begin{proposition}
\label{prop:finiteness-F-compatible}
Let $R$ be a noetherian Frobenius split ring of prime 
characteristic $p > 0$. Then $R$ has finitely many uniformly
$F$-compatible ideals in each of the following cases:
\begin{enumerate}
	\item $R$ is $F$-finite.
	\item $(R,\mathfrak m)$ is local.
\end{enumerate}
In fact, if $\varphi \colon F_*R \rightarrow R$ is a Frobenius
splitting, then $R$ has finitely many $\varphi$-compatible 
ideals in both cases.
\end{proposition}

We will reprove Proposition \ref{prop:finiteness-F-compatible}, in
part to highlight the similarity between the proofs of the global
and local cases, and also to give a proof of the 
global $F$-finite case that does not rely on $F$-adjunction or 
the language of divisor pairs.
But first, we highlight a bound on
the multiplicity of $F$-pure noetherian local
rings that gives a proof of the local case of
Proposition \ref{prop:finiteness-F-compatible}.

\begin{proposition}\cite[Thm.\ 3.1]{HW15}
\label{prop:multiplicity-bound}
Let $(R,\mathfrak m, \kappa)$ be a noetherian local ring of prime 
characteristic $p > 0$ that is $F$-pure. Suppose $d$ is the 
dimension of $R$ and $v = \dim_\kappa \mathfrak{m}/\mathfrak{m}^2$. 
Then
\[
e(R) \leq \binom{v}{d},
\]
where $e(R)$ is the Hilbert--Samuel multiplicity of $R$.
\end{proposition}

\begin{proof}[Proof of Proposition \ref{prop:finiteness-F-compatible}]
Since $R$ is Frobenius split, all uniformly 
$F$-compatible ideals are radical.
Part (1) is precisely \cite[Cor.\ 5.10]{Sch09}. For the 
reader's convenience, we translate Schwede's 
proof, avoiding the machinery of $F$-adjunction and pairs. 
Fix a Frobenius splitting
\[
\varphi: F_*R \rightarrow R.
\]
We will show the stronger statement that there are finitely many
$\varphi$-compatible ideals in the sense of Remark \ref{rem:compatible-fixed-map}. 
As discussed in the beginning of this subsection, it suffices to 
show that there are finitely many 
$\varphi$-compatible ideals that are prime. If not, there will be
infinitely many distinct $\varphi$-compatible prime ideals
$\{\mathfrak{p}_\alpha\}_{\alpha}$ such that $\dim(R/\mathfrak{p}_\alpha) = d$,
for some fixed $0 < d \leq \dim(R)$ ($\dim(R) < \infty$ because $R$ is 
$F$-finite \cite[Prop.\ 1.1]{Kun76}). Let
\[
\mathfrak{q} \coloneqq \bigcap_\alpha \mathfrak{p}_\alpha.
\]
One can verify that if $\mathfrak{p}$ is a minimal prime of $\mathfrak{q}$,
then $\mathfrak{p}$ is the intersection of the $\mathfrak{p}_\alpha$ such that
$\mathfrak{p} \subseteq \mathfrak{p}_\alpha$. Since $\mathfrak{q}$ has
finitely many minimal primes, by the pigeonhole principle, there exists
a minimal prime of $\mathfrak{q}$ that is the intersection of infinitely
many of the $\mathfrak{p}_\alpha$'s. Thus, replacing $\mathfrak{q}$ by this 
minimal prime, we may assume that $\mathfrak{q}$ is prime. Note that for
all $\alpha$, $\mathfrak{q} \subsetneq \mathfrak{p}_\alpha$ since there
are no inclusion relations among the $\mathfrak{p}_\alpha$. 
Moreover, $\mathfrak{q}$ is uniformly
$\varphi$-compatible, so we get an induced Frobenius splitting
\[
\overline{\varphi}: F_*(R/\mathfrak{q}) \rightarrow R/\mathfrak{q}
\]
of the domain $R/\mathfrak{q}$. For all $\alpha$, 
$\mathfrak{p}_\alpha/\mathfrak{q}$ is a nonzero $\overline{\varphi}$-compatible
ideal of $R/\mathfrak{q}$ and
\begin{equation}
\label{eq:intersection}
\bigcap_{\alpha}(\mathfrak{p}_\alpha/\mathfrak{q}) = 
(\bigcap_\alpha \mathfrak{p}_\alpha)/\mathfrak{q}= (0).
\end{equation}
Since $R/\mathfrak{q}$ is an $F$-finite domain and $\overline{\varphi}$ is a nonzero map, 
there is a smallest \emph{nonzero}
$\overline{\varphi}$-compatible ideal with respect to inclusion,
namely the (big or non-finitistic) test ideal $\uptau(R/\mathfrak{q}, \overline{\varphi})$ 
\cite[Thm.\ 3.8]{ST12} (this is where $F$-finiteness is
used seriously for the first time). In particular, 
\[
(0) \neq \uptau(R/\mathfrak{q}, \overline{\varphi}) \subseteq 
\mathfrak{p}_\alpha/\mathfrak{q},
\]
for all $\alpha$. This contradicts (\ref{eq:intersection}).

We prove (2) following \cite[Rem.\ 3.4]{HW15}. As above, for a Frobenius
splitting $\varphi$ of $(R, \mathfrak m)$, it suffices to
show there are only finitely many $\varphi$-compatible prime ideals
$\mathfrak{p}$ of coheight $d$, for any fixed $0 \leq d \leq \dim(R)$.
So suppose 
$\mathfrak{p}_1, \dots, \mathfrak{p}_n \in \Spec(R)$
are prime ideals of coheight $d$. Since $\mathfrak{p}_1 \cap \dots
\cap \mathfrak{p}_n$ is uniformly $F$-compatible, 
\[
R/\bigcap_{i = 1}^n \mathfrak{p}_i
\]
is a Frobenius split equidimensional local 
ring of Krull dimension $d$. Suppose $v$ (resp. $v'$) is the
embedding dimension of $R$ 
(resp. $R/\bigcap_{i = 1}^n \mathfrak{p}_i$). Then $v' \leq v$.
Since $R/\bigcap_{i = 1}^n \mathfrak{p}_i$ is reduced and 
equidimensional, we get
\[
n \leq \sum_{i = 1}^n e(R/\mathfrak{p}_i) = 
e(R/\bigcap_{i = 1}^n \mathfrak{p}_i) \leq \binom{v'}{d} \leq 
\binom{v}{d}.
\]
The first inequality follows because the 
multiplicity of a local domain is a positive integer, the 
equality follows by \cite[Thm.\ 14.7]{Mat89} because $R/\bigcap_{i=1}^n\mathfrak{p}_i$ is
reduced and equidimensional, the second
inequality follows by Proposition \ref{prop:multiplicity-bound}
applied to the $F$-pure local ring $R/\bigcap_{i = 1}^n \mathfrak{p}_i$,
and the final inequality follows because $v' \leq v$. Thus,
for each $0 \leq d \leq \dim(R)$, the number of $\varphi$-compatible
prime ideals of coheight $d$ is
bounded above by $\binom{v}{d}$, which only depends on $R$
and $d$.
\end{proof}

\begin{remark}
{\*}
\begin{enumerate}
\item In the proof of Proposition \ref{prop:finiteness-F-compatible} in the 
$F$-finite case, we used the highly non-trivial fact that
if $R$ is a noetherian $F$-finite domain and $\varphi \colon F_*R \rightarrow R$
is a nonzero $R$-linear map, then $R$ has a \emph{smallest nonzero} $\varphi$-compatible
ideal with respect to inclusion. The existence of this ideal is based on a deep
result of Hochster and Huneke in tight closure theory on the existence of 
completely stable test elements for noetherian $F$-finite rings. For more details,
please see \cite[Lem.\ 3.6]{ST12} and \cite[Thm.\ 5.10]{HH94}.

\vspace{1mm}

\item The finiteness of 
the set of uniformly $F$-compatible ideals of a Frobenius split
noetherian local ring $(R, \mathfrak m)$ in the excellent case also
follows by \cite[Cor.\ 3.2]{EH08}, which is a characteristic independent
result. The advantage of Proposition 
\ref{prop:multiplicity-bound} is that it allows one to obtain explicit bounds
on the number of uniformly $F$-compatible prime ideals 
(and also without any excellence hypothesis). The same
explicit bounds were also obtained in \cite[Thm.\ 4.2]{ST10}  
in the excellent local case, albeit via more involved considerations. It is a different matter
that the authors do not know an example of a non-excellent Frobenius split 
noetherian local ring. 
Indeed, the most common method of constructing non-excellent
noetherian local rings in prime characteristic is in the dimension $1$ 
regular case via \emph{arc valuations}; see \cite{DS18}. However, \cite[Thm.\ 7.4.1]{DMS20}
shows that any Frobenius split normal noetherian domain of dimension $1$
has to be excellent. Thus, the `usual' method of constructing non-excellent
noetherian local rings in prime characteristic do not give examples that
are Frobenius split. Consequently, it is unclear if Proposition 
\ref{prop:finiteness-F-compatible} gives additional cases of finiteness
of $F$-compatible ideals outside the excellent setting in the local case.


\vspace{1mm}

\item The proofs of the finiteness of the set $\Sigma$ of uniformly $F$-compatible
ideals of Frobenius split noetherian rings that are $F$-finite or local 
only used the facts
that $\Sigma$ is closed under arbitrary intersections, consists of  radical ideals, and 
is also closed under taking minimal primes of ideals.
The property of $\Sigma$ being closed under sums of ideals is never used, in contrast with
\cite[Cor.\ 3.2]{EH08} and \cite[Thm.\ 4.1]{ST10}.
\end{enumerate}
\end{remark}

\begin{corollary}
\label{cor:complete-local}
If $(R,\mathfrak m)$ is a complete local noetherian ring of prime 
characteristic that is $F$-pure, then $R$ has finitely many uniformly
$F$-compatible ideals.
\end{corollary}

\begin{proof}
This follows by Proposition \ref{prop:finiteness-F-compatible}, 
because purity of a ring map $R \rightarrow S$, when $R$ is complete, is
equivalent to splitting by a result due to Auslander; see 
\cite[Lem.\ 1.2]{Fed83} (or \cite[Lem.\ 2.3.3]{DMS20} for a more general
statement).
\end{proof}

\section{The splinter locus}
\label{sec:splinter-locus}
Suppose $X$ is a locally 
noetherian scheme. We define
\[
\Spl(X) \coloneqq \textrm{$\{x \in X: \mathcal{O}_{X,x}$ is a 
splinter$\}$}. 
\]
Note that for any open subscheme $U$ of $X$, 
$\Spl(U) = \Spl(X) \cap U$. If $R$ is a noetherian ring, we 
define $\Spl(R)$ to be $\Spl(\Spec(R))$. Since splinters are integrally closed domains, it follows that
$\Spl(X)$ is contained in the normal locus
\[
\Nor(X) \coloneqq \textrm{$\{x \in X: \mathcal{O}_{X,x}$ is an 
integrally closed domain$\}$}
\]
of $X$. Thus, if $\Nor(X)$ is open (for example, if 
$X$ is has an open regular locus \cite[Cor.\ (6.13.5)]{EGAIV_2}), 
then to show that $\Spl(X)$ is open it suffices
to assume that $X$ is normal. Since openness of the 
splinter locus can be checked on a sufficiently fine affine open cover of $X$, 
one may further assume that $X = \Spec(A)$, where $A$
is a noetherian integrally closed domain (in particular, 
$A$ is approximately Gorenstein). Thus, we will 
analyze when the splinter locus of a noetherian domain is open.

\subsection{Traces and splinter loci} 
Recall that if $B$ is an $A$-algebra, then the ideal trace of $B/A$ is
\[
T_{B/A} \coloneqq \bigcap_{I} (I: IB \cap A),
\]
where the intersection ranges over all ideals $I$ of $A$. Similarly, the trace of $B/A$ 
is 
\[
\uptau_{B/A} \coloneqq \im(\Hom_A(B,A) \xrightarrow{\eval \MVAt 1} A).
\]
In general, 
$\uptau_{B/A} \subseteq T_{B/A}$, and equality holds when
$A$ is an approximately Gorenstein noetherian domain and $B$ is a finite
extension of $A$ (Corollary \ref{cor:ideal-traces-approx-gor}(1)). 
We begin with the following observation:

\begin{lemma}
\label{lem:minimal-trace}
Let $A$ be a noetherian domain, $\mathcal{C}$ be the 
collection of finite $A$-subalgebras of $A^+$ and
\[
\textrm{$\Sigma_\uptau \coloneqq \{\uptau_{B/A}: B \in \mathcal{C}\}$ and 
$\Sigma_T \coloneqq \{T_{B/A}: B \in \mathcal{C}\}$}.
\]
\begin{enumerate}
	\item For all $B \in \mathcal{C}$, $\uptau_{B/A} \neq 0$ and
	$T_{B/A} \neq 0$.
	\item If $B, B' \in \mathcal{C}$ such that $B \subseteq B'$,
	then $\uptau_{B'/A} \subseteq \uptau_{B/A}$ (resp.
	$T_{B'/A} \subseteq T_{B/A}$). Thus, $\Sigma_\uptau$ (resp.
	$\Sigma_T$) is a
	cofiltered poset of ideals of $A$ under inclusion.
	\item If a minimal element of $\Sigma_\uptau$ (resp. $\Sigma_T$) exists,
	then this
	is the smallest element of $\Sigma_\uptau$ (resp. $\Sigma_T$) in the sense that it is
	contained in every other element of $\Sigma_\uptau$ (resp. 
	$\Sigma_T$).
	\item If $\Sigma_\uptau$ (resp. $\Sigma_T$) is finite, then it has a smallest
	element under inclusion.
	\item $A$ is a splinter $\Longleftrightarrow$ $\bigcap_{B \in \mathcal
	C} \uptau_{B/A} = A \Longleftrightarrow \bigcap_{B \in \mathcal{C}} T_{B/A} = A$.
\end{enumerate}
\end{lemma}

We say a partially order set $(\Sigma, \leq)$ is \emph{cofiltered} 
if for all $x, y \in \Sigma$, there exists $z \in \Sigma$
such that $z \leq x, y$. 

\begin{proof}
(1) For all $B \in \mathcal C$, $\Hom_A(B,A) \neq 0$ because
\[
\Frac(A) \otimes_A \Hom_A(B, A) = \Hom_{\Frac(A)}(\Frac(B), \Frac(A)) \neq 0.
\]
Thus, for all $B \in \mathcal C$, 
$\uptau_{B/A} \neq 0$. Since $T_{B/A}$ contains $\uptau_{B/A}$ 
(Lemma \ref{lem:contraction-ideals}(3)),
it follows that for all $B \in \mathcal{C}$, $T_{B/A} \neq 0$.

(2) If $B \subseteq B'$, then $\Hom_A(B',A) \xrightarrow{\eval
 \MVAt 1} R$ factors as
\[
\Hom_A(B', A) \rightarrow \Hom_A(B, A) \xrightarrow{\eval \MVAt 1} A,
\]
where $\Hom(B',A) \rightarrow \Hom(B,A)$ is given by restriction
to $B$. It follows that $\uptau_{B'/A} \subseteq \uptau_{B/A}$. 
The set $\mathcal{C}$ is a filtered poset under inclusion
because if $B, B' \in \mathcal C$, then $B[B'] \in \mathcal C$.
Thus,  $\Sigma_\uptau$ is a cofiltered poset
under inclusion.

Analogously, if $B \subseteq B'$, then $IB \cap A \subseteq IB' \cap A$. 
This means that $(I: IB'\cap A) \subseteq (I: IB \cap A)$, and
so, $T_{B'/A} \subseteq T_{B/A}$. Again, because $\mathcal{C}$ is a
filtered poset, $\Sigma_T$ is a cofiltered poset.

(3) holds generally for any cofiltered poset.

(4) follows from (3) because any finite poset has a minimal
element.

(5) If $A$ is a splinter, then $\uptau_{B/A} = A$ for any finite
extension of $A$ by definition of the splinter property. 
This shows that if $A$ is a splinter, then $\bigcap_{B \in \mathcal
	C} \uptau_{B/A} = A$.

If $\bigcap_{B \in \mathcal{C}} \uptau_{B/A} = A$, then
$\bigcap_{B \in \mathcal{C}} T_{B/A} = A$ because $\uptau_{B/A} 
\subseteq T_{B/A}$. 

It remains to show that if 
$$\bigcap_{B \in \mathcal{C}} T_{B/A} = A,$$ 
then $A$ is splinter. If $1 \in T_{B/A}$,
then for all ideals $I$ of $A$ we have 
$$I = IB \cap A.$$ 
Thus
$A \rightarrow B$ is cyclically pure for
all $B \in \mathcal{C}$. Since $A^+ = \colim_{B \in \mathcal{C}} B$,
it follows that $A \rightarrow A^+$ is also cyclically pure.
Then $A$ is a splinter by \cite[Lem.\ 2.3.1]{DT19}. 
\end{proof}

\begin{remark}
{\*}
\begin{enumerate}
    \item Using the notation of Lemma \ref{lem:minimal-trace}, if 
    $\Sigma_\uptau$ has a smallest element under inclusion, then there
    exists a finite $A$-subalgebra $B_0$ of $A^+$ such that
    \[
    \uptau_{B_0/A} = \bigcap_{B \in \mathcal{C}} \uptau_{B/A}.
    \]
    This implies that for all finite $A$-subalgebras $B'$ of $A^+$
    containing $B_0$,
    \[
    \uptau_{B_0/A} = \uptau_{B'/A},
    \]
    that is, the traces of the finite $A$-subalgebras
    of $A^+$ stabilize. We will see that this stable trace ideal,
    when it exists, defines the splinter locus of $A$.

    \item A similar stabilization result also holds if $\Sigma_T$
    has a smallest element under inclusion. If $A$
    is additionally approximately Gorenstein, then 
    $\Sigma_\uptau = \Sigma_T$ by Corollary \ref{cor:test-ideal-trace}(1),
    that is, we get the same stable ideal as in the previous remark.
    
\end{enumerate}
\end{remark}

\begin{definition}
\label{def:splinter-ideal}
Let $A$ be a noetherian domain and $\mathcal{C}$ be the 
collection of finite $A$-subalgebras of $A^+$. We define the
\emph{trace of $A$}, denoted $\uptau_A$, to be
\[
\uptau_A \coloneqq \bigcap_{B \in \mathcal{C}} \uptau_{B/A}.
\]
The \emph{ideal trace of $A$}, denoted $T_A$, is defined to be
\[
T_A \coloneqq \bigcap_{B \in \mathcal{C}} T_{B/A}.
\]
\end{definition}

\begin{proposition}
\label{prop:splinter-ideal-properties}
Let $A$ be a noetherian domain and $\mathcal{C}$ be the 
collection of finite $A$-subalgebras of $A^+$. 
Then we have the following:
\begin{enumerate}
    \item $A$ is a splinter $\Longleftrightarrow \uptau_A = A \Longleftrightarrow T_A = A$.
    \item $T_A = T_{A^+/A}$. 
    \item If $A$ is approximately Gorenstein, then $\uptau_A = T_A$
    and $\uptau_A$ is the big plus closure test ideal.
    \item If $(A,\mathfrak m)$ is complete local, 
    then $\uptau_A = \uptau_{A^+/A}$.
\end{enumerate}
Assume that $\Sigma_\uptau \coloneqq \{\uptau_{B/A} \colon B \in \mathcal{C}\}$
has a smallest element under inclusion. Then:
\begin{enumerate}
\setcounter{enumi}{4}
    \item There exists $B_0 \in \mathcal{C}$
    such that $\uptau_A = \uptau_{B_0/A}$. 
    \item If $A$ is approximately Gorenstein, there exists $B_0 \in \mathcal{C}$
    such that $T_A = T_{B_0/A} = \uptau_{B_0/A}$.
     \item If $(A, \mathfrak m)$ is complete local, 
    there exists $B_0 \in \mathcal{C}$ such that $\uptau_{A^+/A} = 
    \uptau_{B_0/A} = T_{B_0/A} = T_{A^+/A}$.
    \item If $\mathfrak{p} \in \Spec(A)$, then $\uptau_{A_\mathfrak{p}} = (\uptau_A)_{\mathfrak{p}}$.
    \item For $\mathfrak{p} \in \Spec(A)$, $A_{\mathfrak p}$ is a 
    splinter if and only if $\uptau_A \nsubseteq \mathfrak{p}$. 
    Thus, the splinter locus of $A$ is the complement in $\Spec(A)$ of $\mathbf{V}(\uptau_A)$.
    \item There exists a finite $A$-subalgebra $B_0$ of $A^+$ such that if
    $A \hookrightarrow B_0$ splits, then $A$ is a splinter.
\end{enumerate}
\end{proposition}

\begin{proof}
(1) is precisely Lemma \ref{lem:minimal-trace}(5). 

(2) For all $B \in \mathcal{C}$, we have $B \subseteq A^+$. Therefore,
\[
T_{A^+/A} = \bigcap_I (I:IA^+ \cap A) \subseteq \bigcap_I (I:IB\cap A) = T_{B/A},
\]
where the intersections are indexed by all ideals $I$ of $A$ and the middle containment 
follows because $IB \cap A \subseteq IA^+ \cap A$. Thus, 
$$T_{A^+/A} \subseteq T_A.$$
Now suppose $c \in T_A$ and let $I$ be an ideal of $A$. 
To prove $T_A \subseteq T_{A^+/A}$
we have to show that $c(IA^+ \cap A) \subseteq I$.
Let $z \in IA^+ \cap A$ and choose a finite $A$-subalgebra $B$ of $A^+$
such that $z \in IB \cap A$. Since $c \in T_{B/A}$, it follows that
$cz \in I$, and so, $c(IA^+ \cap A) \subseteq I$.


(3) If $A$ is approximately Gorenstein, then for all $B \in 
\mathcal{C}$, 
\[
\uptau_{B/A} = T_{B/A}
\]
by Corollary \ref{cor:ideal-traces-approx-gor}(1). So $\uptau_A = T_A$
by the definition of these ideals. By (2), we then get $\uptau_A = T_{A^+/A}$, 
and taking $B = A^+$ in 
Proposition \ref{prop:algebra-closure}(4), it follows that 
$T_{A^+/A} = \uptau_{\cl_{A^+}}(A)$. Since $\cl_{A^+}$ is precisely
plus closure, $\uptau_{\cl_{A^+}}(A)$ is the big plus closure test
ideal.

(4) A noetherian complete local domain is approximately Gorenstein
because reduced excellent rings are approximately Gorenstein. Then
$\uptau_A = T_A$ by (3), $T_A = T_{A^+/A}$ by (2), and
$T_{A^+/A} = \uptau_{A^+/A}$ by Corollary \ref{cor:ideal-traces-approx-gor}(4).

(5) follows by
the hypothesis that $\Sigma_\uptau$ has a smallest element.

(6) Since $A$ is approximately Gorenstein, 
(3) implies
\[
T_A = \uptau_A.    
\]
Now choose $B_0 \in \mathcal{C}$ that satisfies the conclusion of (5).
Then 
$T_A = \uptau_A = \uptau_{B_0/A} = T_{B_0/A}$,
where the last equality
again follows by the approximately Gorenstein property and 
Corollary \ref{cor:ideal-traces-approx-gor}(1).


(7) Since complete local domains are
approcimately Gorenstein, by (2), (3) and (4), 
$$\uptau_{A^+/A} = \uptau_{A} = T_A = T_{A^+/A},$$ 
and by (6), there exists
$B_0 \in \mathcal{C}$ such that $T_A = T_{B_0/A} = \uptau_{B_0/A}$. 

(8) If $\Sigma_\uptau$ has a smallest element $\uptau_{B_0/A}$, then by exactness of
localization, for $\mathfrak p \in \Spec(A)$, the collection
\[
\Sigma_{\uptau, \mathfrak p} \coloneqq \{(\uptau_{B/A})_{\mathfrak p}: B \in \mathcal{C}\}
\]
also has a smallest element under inclusion, namely $(\uptau_{B_0/A})_{\mathfrak p} = 
\uptau_{{(B_0)}_{\mathfrak p}/A_{\mathfrak p}}$. Here the last equality follows by
Lemma \ref{lem:trace-ideals}(3). Thus,
\[
(\uptau_A)_{\mathfrak p} = (\bigcap_{B \in \mathcal{C}} \uptau_{B/A})_{\mathfrak p} = 
(\uptau_{B_0/A})_{\mathfrak p} 
= \bigcap_{B \in \mathcal C}(\uptau_{B/A})_{\mathfrak p} = 
\bigcap_{B \in \mathcal{C}}\uptau_{B_{\mathfrak p}/A_{\mathfrak p}}.
\]
It suffices to show that 
\begin{equation}
\label{eq:key-eq}
\bigcap_{B \in \mathcal{C}}\uptau_{B_{\mathfrak p}/A_{\mathfrak p}} = \uptau_{A_\mathfrak{p}}.
\end{equation}
Let $T$ be a finite $A_{\mathfrak p}$-subalgebra of 
$(A_\mathfrak{p})^+ = (A^+)_{\mathfrak p}$. Suppose $T = A_{\mathfrak p}[t_1,\dots,t_n]$. Since each
$t_i$ is integral over $A_{\mathfrak p}$, there exists $s \in A \setminus \mathfrak{p}$
such that for all $i$, $st_i$ is integral over $A$. As $s$ is a unit in $T$, 
replacing $t_i$ by $st_i$ does not change $T$, so 
we may assume that each $t_i$ is integral over $A$. Then
the $A$-subalgebra
\[
T' \coloneqq A[t_1,\dots,t_n]
\]
of $T$ has the property that $T'$ is a finite extension of $A$ contained in $A^+$ and
$(T')_{\mathfrak p} = T.$
Thus, 
\[
\uptau_{T/A_{\mathfrak p}} = \uptau_{(T')_{\mathfrak p}/A_{\mathfrak p}} = 
(\uptau_{T'/A})_{\mathfrak p}.
\]
Said differently, this argument shows 
that every trace of a finite $A_{\mathfrak p}$-subalgebra of
$(A_{\mathfrak p})^+$ is the localization at $\mathfrak p$ of the trace of some finite
$A$-subalgebra of $A^+$. Unravelling the definition of $\uptau_{A_{\mathfrak{p}}}$,
this implies (\ref{eq:key-eq}).

(9) By (1), $A_{\mathfrak p}$ is a splinter if and only if $\uptau_{A_\mathfrak{p}} = A_{\mathfrak p}$. Now by (7), $\uptau_{A_\mathfrak{p}} = (\uptau_A)_{\mathfrak p}$. Thus,
$\uptau_{A_\mathfrak{p}} = A_{\mathfrak p}$ if and only if $(\uptau_A)_{\mathfrak p} = A_{\mathfrak p}$, 
and this last equality holds precisely when
$\uptau_A \nsubseteq \mathfrak{p}$. It follows that the splinter locus of $A$ is the
complement of the closed set $\mathbf{V}(\uptau_A)$ of $\Spec(A)$.

(10) Choose a finite $A$-subalgebra $B_0$ of $A^+$ that satisfies (5). 
If $A \hookrightarrow B_0$ splits, then $\uptau_A = \uptau_{B_0/A} = A$, and so,
$A$ is a splinter by (1).
\end{proof}

\begin{remark}
The equality $\uptau_A = \uptau_{A^+/A}$ in part (4) and the equality
$\uptau_{A^+/A} = \uptau_{B_0/A}$ in part (6) of Proposition \ref{prop:splinter-ideal-properties} fail quite dramatically if we do
not assume $A$ is complete, even for excellent regular local rings. Indeed, by 
\cite{DM20} choose an excellent
Henselian regular local ring $A$ of Krull dimension $1$ and prime characteristic 
$p > 0$ such that $A$ admits no nonzero $A$-linear maps $F_*A \rightarrow A$. 
Since $F_*A$
embeds in $A^+$, this implies that there are no nonzero $A$-linear maps 
$A^+ \rightarrow A$. Thus, $\uptau_{A^+/A} = 0$ while $\uptau_A = A$ because 
regular rings of prime characteristic are splinters by \cite{Hoc73}. Note that in 
this case $\uptau_A$ equals $\uptau_{B/A}$ for any finite $A$-subalgebra $B$ of $A^+$.
\end{remark}

\begin{example}
Suppose $R$ is a noetherian $F$-finite normal domain that is $\mathbf{Q}$-Gorenstein.
Then \cite{BST15} shows that the set $\Sigma_\uptau$ of 
Proposition \ref{prop:splinter-ideal-properties} has a smallest element,
namely the big test ideal $\uptau_b(R)$ of $R$. Thus, $\uptau_R = \uptau_b(R)$,
and so, Proposition \ref{prop:splinter-ideal-properties}(9) shows that
the splinter locus of $R$ coincides with the complement of $\mathbf{V}(\uptau_b(R))$,
which is the strongly $F$-regular locus of $R$. This is not surprising because 
Singh showed that in the affine $\mathbf{Q}$-Gorenstein setting, the splinter
condition is the same as being $F$-regular \cite{Sin99(b)} and it is known that
for $F$-finite $\mathbf{Q}$-Gorenstein rings, 
(weak) $F$-regularity is equivalent to strong $F$-regularity \cite{Mac96}.
\end{example}

We will now use Proposition \ref{prop:splinter-ideal-properties} to compare the traces
under Henselizations and completions.

\begin{proposition}
\label{prop:trace-henselization-completion}
Let $(A, \mathfrak m)$ be a noetherian normal domain. We have the following:
\begin{enumerate}
    \item If $A^h$ is the Henselization of $A$ with respect to $\mathfrak{m}$, then
    $\uptau_{A^h} \cap A = \uptau_A$.
    \item If $A$ has geometrically regular formal fibers and $\widehat{A}$ is the $\mathfrak m$-adic completion, then $\uptau_{\widehat{A}} \cap A 
    = \uptau_A$.
\end{enumerate}
\end{proposition}

We need the following lemma, which is interesting in its own right.

\begin{lemma}
\label{lem:descent-plus-closure-test-elements}
Let $A \rightarrow B$ be a cyclically pure map of noetherian domains. Then $T_B \cap A \subseteq T_A$.
\end{lemma}

\begin{proof}[Proof of Lemma \ref{lem:descent-plus-closure-test-elements}]
By Proposition \ref{prop:splinter-ideal-properties}(2), $T_A = T_{A^+/A}$ and
$T_B = T_{B^+/B}$. Cyclic purity implies $A \rightarrow B$ is injective. So
we may assume that $A \subseteq B$ and $A^+ \subseteq B^+$.

Let $c \in T_B \cap A = T_{B^+/B} \cap A$, and pick any ideal $I$ of $A$. Then
\[
c(IA^+ \cap A) \subseteq c((IB)B^+ \cap B) \cap A \subseteq IB \cap A = I,
\]
where the first containment follows because $IA^+ \cap A \subseteq (IB)B^+ \cap B$,
the second containment because $c \in T_{B^+/B}$, and the equality because $A \rightarrow B$
is cyclically pure. Thus,
$c \in \bigcap_I (I: IA^+ \cap A) = T_{A^+/A} = T_A$.
\end{proof}

We now prove Proposition \ref{prop:trace-henselization-completion}
utilizing some ideal-theoretic results of \cite{DT19}.

\begin{proof}[Proof of Proposition \ref{prop:trace-henselization-completion}]
Note that a noetherian integrally closed domain is approximately Gorenstein.

(1) $A^h$ is also a
noetherian integrally closed domain 
\cite[\href{https://stacks.math.columbia.edu/tag/06DI}{Tag 06DI}]{stacks-project}. Thus, by Proposition 
\ref{prop:splinter-ideal-properties}(3) and (2), 
\[
\textrm{$\uptau_A = T_{A^+/A}$ \hspace{2mm} and \hspace{2mm} $\uptau_{A^h} = T_{(A^h)^+/A^h}$}.
\]
Therefore it suffices to show that $T_{(A^h)^+/A^h} \cap A = T_{A^+/A}$. Since $A \rightarrow A^h$ is
faithfully flat, Lemma \ref{lem:descent-plus-closure-test-elements}
and Proposition \ref{prop:splinter-ideal-properties}(2) give
\[
T_{(A^h)^+/A^h} \cap A \subseteq T_{A^+/A}.
\]
Now suppose $c \in T_{A^+/A}$. It remains to show that $c \in T_{(A^h)^+/A}$. If $\mathcal{C}$
is the collection of $\mathfrak{m}A^h$-primary ideals of $A^h$, then Proposition \ref{prop:algebra-closure}(5)
shows that
\[
T_{(A^h)^+/A^h} = \bigcap_{J \in \mathcal{C}} (J: J(A^h)^+ \cap A^h).
\]
Therefore it suffices to show that $c(J(A^h)^+ \cap A^h) \subseteq J$, for any $\mathfrak{m}A^h$-primary
ideal $J$ of $A^h$. Note that 
\[
A \rightarrow A^h
\]
is a local homomorphism such that the induced map on completions is an isomorphism. Therefore 
any $\mathfrak{m}A^h$-primary ideal of $A^h$ is expanded from an $\mathfrak m$-primary ideal of $A$
(for example, see \cite[Lem.\ 3.1.2]{DT19}). So choose an $\mathfrak{m}$-primary ideal $I$ of $A$
such that
\[
J = IA^h.
\]
By \cite[Prop. 3.1.4(2)]{DT19},
\[
(J(A^h)^+ \cap A^h) \cap A = J(A^h)^+ \cap A =  I(A^h)^+ \cap A = IA^+ \cap A.
\]
Moreover, $J \subseteq J(A^h)^+ \cap A^h$, which means that $J(A^h)^+ \cap A^h$ is
also an $\mathfrak{m}A^h$-primary ideal expanded from some $\mathfrak m$-primary ideal
of $A$. Then it must be the case that
\[
J(A^h)^+ \cap A^h = (IA^+ \cap A)A^h.
\]
Since $c \in T_{A^+/A}$, we have $c(IA^+ \cap A) \subseteq I$. Consequently,
\[
c(J(A^h)^+ \cap A^h) = c\big{(}(IA^+ \cap A)A^h\big{)} = (c(IA^+ \cap A))A^h
\subseteq IA^h = J,
\]
as desired.

(2) Since $A$ is a normal domain with geometrically regular formal fibers,
$\widehat{A}$ is also a normal domain 
\cite[\href{https://stacks.math.columbia.edu/tag/0BFK}{Tag 0BFK}]{stacks-project} 
and both $A, \widehat{A}$ are 
approximately Gorenstein. Now the rest of the proof of (2) follows from
the argument given in (1) but with $A^h$ replaced by $\widehat{A}$,
$(A^h)^+$ replaced by $\widehat{A}^+$, and \cite[Prop. 3.1.4(2)]{DT19}
replaced by \cite[Prop. 3.2.2]{DT19}, which says that for an ideal $I$
of $A$, $I\widehat{A}^+ \cap A = IA^+ \cap A$.
\end{proof}

\subsection{Separable traces and splinter loci} Let $R$ be a 
noetherian domain with fraction field $K$. Recall that we say $R$ is \emph{N-1} if
the integral closure of $R$ in $K$ is a finite $R$-algebra. 


Excellent domains are \emph{N-1}, although
the \emph{N-1} assumption is substantially more general. 
For example, any noetherian normal
domain is \emph{N-1}, although noetherian normal domains are far from being 
excellent in general. Moreover, in the context of singularity
theory, especially in prime characteristic, most notions of
$F$-singularities such as $F$-injective, $F$-pure, Frobenius
split, splinter, $F$-rational and all avatars of $F$-regular 
imply the \emph{N-1} property at the level of local rings 
\cite[Lem.\ 7.1.4]{DMS20}.

If $R$ is a domain, then we will use $(R^+)^{\s}$ to denote 
the subring of $R^+$ consisting of those elements whose minimal polynomials
over $K$ are separable. Thus, if $K^{\s}$ is
the maximal separable extension of $K$ in $\overline{K} = \Frac(R^+)$, then
$(R^+)^{\s}$ is the integral closure of $R$ in $K^{\s}$, or equivalently,
 $(R^+)^{\s} = R^+ \cap K^{\s}$.

We will now specialize to the setting where $R$ is a noetherian domain of 
prime characteristic $p > 0$. Then recall that for an ideal $I$ of $R$, 
the \emph{Frobenius closure of $I$}, denoted $I^{[F]}$, is defined as
\[
I^{[F]} \coloneqq \textrm{$\{r \in R \colon r^{p^e} \in I^{[p^e]}$ for some 
$e \in \mathbf{Z}_{\geq 0}$\}}.
\]
One can verify that $I^{[F]}$ is an ideal of $R$ that is contained in the tight
closure $I^*$.

The following Proposition, proved by Singh \cite{Sin99(a)}, is the key result that motivates
this section (see also \cite{SS12}). We state the Proposition with a more general hypothesis than in \emph{loc.~cit.},
and explain why Singh's arguments only need this weaker hypothesis.

\begin{proposition}\cite{Sin99(a)}
\label{prop:singh-separable}
Let $R$ be a noetherian N-1 domain of prime characteristic $p > 0$ and 
fraction field $K$. Let $I$ be an ideal of $R$.
\begin{enumerate}
    \item If $r \in I^{[F]}$, then there exists
    a finite generically \'etale $R$-subalgebra $S$ of $R^+$ such that 
    $r \in IS$.
    
    \item For all ideals $I$ of $R$, 
    $IR^{+} \cap R = I(R^+)^{\s} \cap R$.
\end{enumerate}
\end{proposition}

\begin{proof}[Indication of proof]
Singh proves (1) in \cite[Thm.\ 3.1]{Sin99(a)} assuming that $R$ is excellent. 
In the proof, he considers roots $u_1,\dots,u_n$ of certain Artin-Schreier polynomials
over $K$ and then takes $S$ to be the integral closure of $R$ in the fraction
field $L$ of $R[u_1,\dots,u_n]$. Note that $L$ is a finite separable extension
of $K$ by construction. The only place excellence appears to be used 
in the proof of \emph{loc.~cit} is to conclude that $S$ is a finite $R$-algebra. We 
claim this follows as long as $R$ is \emph{N-1}. Indeed, if $\overline{R}$
is the integral closure of $R$ in $K$, then $S$ is also the integral
closure of $\overline{R}$ in $L$. Since $\overline{R}$ is a normal noetherian
domain (it is module finite over $R$ by the \emph{N-1} hypothesis), $S$ is then
a finite $\overline{R}$-algebra by 
\cite[\href{https://stacks.math.columbia.edu/tag/032L}{Tag 032L}]{stacks-project}. 
The point here is that since $L/K$ is a finite separable extension, it admits
a nonzero trace $\Tr_{L/K}$ that restricts to give a nonzero $\overline{R}$-linear map 
$S \rightarrow \overline{R}$ because $\overline{R}$ is normal. However, 
any generically finite solid algebra extension of noetherian
domains is actually finite \cite[Prop.\ 3.7]{DS18}.
Consequently, $S$ is a finite $R$-algebra, and the rest
of the proof of \cite[Thm.\ 3.1]{Sin99(a)} applies without change.

Similarly, (2) is \cite[Cor.\ 3.4]{Sin99(a)}, but again where $R$ is assumed to be
excellent. As in \emph{loc.~cit.},
for $z \in IR^+ \cap R$ 
choose a finite $R$-subalgebra $R_1$ of $R^+$ such that
$z \in IR_1$. If $R_2$ is the largest separable extension of $R$ is $R_1$,
then $z \in (IR_2)^{[F]}$ because $R_2 \hookrightarrow R_1$ is purely
inseparable and module finite. 
Let $L = \Frac(R_2)$, which is a finite separable extension of $K$. 
Since $R$ is \emph{N-1}, the argument in the
previous paragraph shows that the integral closure of $R$ in $L$ is a finite
$R$-algebra. But this integral closure is also the integral 
closure of $R_2$ in $L = \Frac(R_2)$. Thus, $R_2$ is \emph{N-1}. Now by (1), one can
find a finite generically \'etale $R_2$-subalgebra of 
$(R_2)^+ = R^+$ such that $z \in IS$.
Then $R \hookrightarrow S$ is a finite generically \'etale extension, and we
are done.
\end{proof}

As a consequence, we obtain the following characterization of splinters in
prime characteristic for \emph{N-1} domains without any excellence or approximately Gorenstein
hypotheses.

\begin{corollary}{\normalfont (c.f. \cite[Cor.\ 3.9]{Sin99(a)})}
\label{cor:Singh-splinter-sep}
Let $R$ be a noetherian N-1 domain of prime characteristic $p>0$.
Then $R$ is a splinter if and only if $R \hookrightarrow S$ is cyclically pure
for every generically \'etale finite extension domain $S$.
\end{corollary}

\begin{proof}
The `if' implication is the non-trivial one. If $R \hookrightarrow S$ is 
cyclically pure for every generically \'etale finite extension domain
$S$, then $R \rightarrow (R^+)^{\s}$ is cyclically pure because $(R^+)^{\s}$ is a
filtered union of generically \'etale finite $R$-subalgebras. Now since $R$ is
\emph{N-1}, by Proposition 
\ref{prop:singh-separable}(2), $R \rightarrow R^+$ is cyclically pure. 
Then $R$ is a splinter by \cite[Lem.\ 2.3.1]{DT19}. 

Again, Singh proves the same result assuming $R$ is excellent \cite[Cor.\ 3.9]{Sin99(a)} and that $R$ is a direct summand of every generically \'etale 
finite extension domain $S$.
\end{proof}

Motivated by Corollary \ref{cor:Singh-splinter-sep} we introduce the
following definition.

\begin{definition}
\label{def:separable-trace}
Let $R$ be a noetherian domain and $\mathcal{C}^{\s}$ be the collection of 
generically \'etale finite $R$-subalgebras of $R^+$. We define the 
\emph{separable trace of $R$}, denoted $\uptau^{\s}_R$, to be
\[
\uptau^{\s}_R \coloneqq \bigcap_{S \in \mathcal{C}^{\s}} \uptau_{S/R}.
\]
The \emph{separable ideal trace} of $R$, denoted $T^{\s}_R$, is
\[
T^{\s}_R \coloneqq \bigcap_{S \in \mathcal{C}^{\s}} T_{S/R}.
\]
\end{definition}

\begin{remark}
\label{rem:separable-trace-observations}
{\*}
\begin{enumerate}
    \item If $\Frac(R)$ has characteristic $0$, then $\uptau^{\s}_R = \uptau_R$
    and $T^{\s}_R = T_R$.
    \item If $\Frac(R)$ has characteristic $p > 0$, we have $\uptau_R \subseteq \uptau^{\s}_R$ and $T_R \subseteq T^{\s}_R$.
    \item In general, $\uptau^{\s}_R \subseteq T^{\s}_R$. 
    \item Since $\mathcal{C}^{\s}$ is a filtered poset under inclusion,
    the same argument as in Lemma \ref{lem:minimal-trace} demonstrates that
    \[
    \Sigma^{\s}_\uptau \coloneqq \{\uptau_{S/R} \colon S \in \mathcal{C}^{\s}\}
    \]
    and
    \[
    \Sigma^{\s}_T \coloneqq \{T_{S/R} \colon S \in \mathcal{C}^{\s}\}
    \]
    are cofiltered collections of ideals of $R$ under inclusion. 
    In particular, if $\Sigma^{\s}_\uptau$ (resp. $\Sigma^{\s}_T$) has a minimal element, 
    then it has a smallest element.
\end{enumerate}
\end{remark}

We then have the following analogue of Proposition \ref{prop:splinter-ideal-properties}. We state it for noetherian domains of prime
characteristic because for mixed characteristic and equal characteristic
$0$ domains, the separable trace provides no new information over the usual trace.

\begin{proposition}
\label{prop:sep-splinter-ideal-properties}
Let $R$ be a noetherian N-1 domain of prime characteristic $p > 0$.  
Let $\mathcal{C}^{\s}$ be the 
collection of generically \'etale finite $R$-subalgebras of $R^+$. 
Then we have the following:
\begin{enumerate}
    \item $R$ is a splinter $\Longleftrightarrow \uptau^{\s}_R = R
    \Longleftrightarrow T^{\s}_R = R$.
    \item $T^{\s}_{R} = T_{(R^+)^{\s}/R}$. 
    \item If $R$ is approximately Gorenstein, then $\uptau^{\s}_R
    = T^{\s}_R$ and $\uptau^{\s}_R$ is the big separable plus closure
    test ideal.
    \item If $(R,\mathfrak m)$ is complete local, 
    then 
    $\uptau^{\s}_R = \uptau_{(R^+)^{\s}/R}$.
\end{enumerate}
Assume $\Sigma^{\s}_\uptau \coloneqq 
\{\uptau_{S/R} \colon S \in \mathcal{C}^{\s}\}$
has a smallest element under inclusion. Then:
\begin{enumerate}
\setcounter{enumi}{4}
    \item There exists $S_0 \in \mathcal{C}^{\s}$
    such that $\uptau^{\s}_R = \uptau_{S_0/R}$.
    \item If $R$ is approximately Gorenstein, there exists $B_0 
    \in \mathcal{C}^{\s}$ such that 
    $T^{\s}_R = T_{S_0/R} = \uptau_{S_0/R}$.
     \item If $(R, \mathfrak m)$ is complete local, 
    there exists $S_0 \in \mathcal{C}^{\s}$ such that $\uptau_{(R^+)^{\s}/R} = 
    \uptau_{S_0/R} = T_{S_0/R} = T_{(R^+)^{\s}/R}$.
    \item If $\mathfrak{p} \in \Spec(R)$, then $\uptau^{\s}_{R_\mathfrak{p}} = (\uptau^{\s}_R)_{\mathfrak{p}}$.
    \item For $\mathfrak{p} \in \Spec(R)$, $R_{\mathfrak p}$ is a 
    splinter if and only if $\uptau^{\s}_R \nsubseteq \mathfrak{p}$. Thus, the splinter locus of $R$ is the complement in $\Spec(R)$ of $\mathbf{V}(\uptau^{\s}_R)$.
    \item There exists a finite generically \'etale $R$-subalgebra $S_0$ of $R^+$ such that if
    $R \hookrightarrow S_0$ splits, then $R$ is a splinter.
\end{enumerate}
\end{proposition}

\begin{proof}[Sketch of proof]
(1) follows by Corollary \ref{cor:Singh-splinter-sep}.

(2) follows
using the proof of Proposition  \ref{prop:splinter-ideal-properties}(2) verbatim,
but with $A$ (resp. $A^+$)
replaced by $R$ (resp. $(R^+)^{\s}$).

(3) The equality $\uptau^{\s}_R = T^{\s}_R$ follows by
Corollary \ref{cor:ideal-traces-approx-gor}(1). Then $\uptau^{\s}_R$
is the big separable plus closure test ideal because it equals
$T_{(R^+)^{\s}/R}$ using (2), and the latter ideal is the
big separable plus closure test ideal by
Proposition \ref{prop:algebra-closure}(4)
applied to $B = (R^+)^{\s}$.

(4) A complete local domain is approximately Gorenstein. Therefore
$\uptau^{\s}_R = T^{\s}_R = T_{(R^+)^{\s}/R}$ by (2) and (3), and
$T_{(R^+)^{\s}/R} = \uptau_{(R^+){\s}/R}$ by
Corollary \ref{cor:ideal-traces-approx-gor}(3) applied to $B = (R^+)^{\s}$,
which is a solid $R$-algebra because $R^+$ is a solid $R$-algebra.

(5) follows because $\Sigma^{\s}_\uptau$ has a smallest element.

(6) follows by (2), (3), (5) and Corollary \ref{cor:ideal-traces-approx-gor}(1) 
because $\uptau_{S_0/R} = T_{S_0/R}$.

(7) follows by (2), (3), (4) and (6).

(8) follows using the same line of reasoning as in 
Proposition \ref{prop:splinter-ideal-properties}(8). 
The key point is that
if $\Sigma^{\s}_\uptau$ has a smallest element, then the a priori infinite
intersection
\[
\uptau^{\s}_R = \bigcap_{S \in \mathcal{C}^{\s}} \uptau_{S/R}
\]
behaves like a finite intersection, and hence it commutes with localization at $\mathfrak p$,
giving
\[
(\uptau^{\s}_R)_{\mathfrak p} = 
\bigcap_{S \in \mathcal{C}^{\s}} \uptau_{S_{\mathfrak p}/R_{\mathfrak p}}.
\]
One can then show $\bigcap_{S \in \mathcal{C}^{\s}} \uptau_{S_{\mathfrak p}/R_{\mathfrak p}} = 
\uptau^{\s}_{R_{\mathfrak p}}$
by a similar spreading out argument. Indeed, the $R$-subalgebra $T'$ of 
$T$ constructed in the proof of 
Proposition \ref{prop:splinter-ideal-properties}(8) will
be generically \'etale if $T$ is a generically \'etale finite 
$R_{\mathfrak p}$-subalgebra of $((R_{\mathfrak p})^+)^{\s} = ((R^+)^{\s})_{\mathfrak p}$.

(9) follows from (8) and (1) and because the property of
being \emph{N-1} localizes 
\cite[\href{https://stacks.math.columbia.edu/tag/032G}{Tag 032G}]{stacks-project}. 

Finally for (10), any $S_0$ satisfying the conclusion
of (5) works by (1).
\end{proof}

We obtain the following non-obvious consequence of the previous 
results.

\begin{corollary}
\label{cor:the-two-traces}
Let $R$ be a noetherian N-1 domain of prime 
characteristic $p > 0$ that is approximately Gorenstein.
Then
\[
\uptau_R = \uptau^{\s}_R.
\]
\end{corollary}

\begin{proof}
By Proposition \ref{prop:splinter-ideal-properties} parts (2) and (3)
and Proposition \ref{prop:sep-splinter-ideal-properties} parts 
(2) and (3),
we have
\[
\textrm{$\uptau_R = T_{R^+/R}$ \hspace{2mm} and \hspace{2mm}
$\uptau^{\s}_R = T_{(R^+)^{\s}/R}$}.
\]
Now by Proposition \ref{prop:singh-separable}(2),
\[
T_{R^+/R} = \bigcap_I (I: IR^+ \cap R) = 
\bigcap_I (I: I(R^+)^{\s} \cap R) = T_{(R^+)^{\s}/R},
\]
where the intersections range over all ideals $I$ of $R$. This completes the proof.
\end{proof}

\subsection{Openness of splinter loci in prime characteristic} 
We will now show that the splinter locus is open for schemes in prime 
characteristic that are of most interest in arithmetic and geometry. 
In particular, we will show that the splinter locus of any scheme of finite
type over an excellent local ring of prime characteristic is open.
In fact, our results will hold more generally for some quasi-excellent 
schemes and even some schemes that are not quasiexcellent (see Remarks 
\ref{rem:weaking-formal-fibers} and \ref{rem:open-F-pure-loci}). 
Recall that a noetherian ring $R$ is \emph{quasi-excellent} if
the local rings of $R$ have geometrically regular formal fibers and if 
the regular locus of any finite type $R$-algebra is open. Thus an
excellent ring is a quasi-excellent ring that is universally catenary.
Our first result is affine in nature.

\begin{theorem}
\label{thm:splinter-F-split}
Let $R$ be a noetherian $F$-pure domain of prime characteristic $p > 0$ and assume
that $R$ satisfies any of the following conditions:
\begin{enumerate}
    \item[(i)] $R$ is $F$-finite.
    \item[(ii)] $R$ is local (not necessarily excellent).
    \item[(iii)] $(A, \mathfrak m)$ is a noetherian local ring of prime 
    characteristic 
    $p > 0$ with geometrically regular formal 
    fibers and $R$ is essentially of finite type over $A$.
    
\end{enumerate}
Let $\mathcal{C}$ be the collection of finite $R$-subalgebras of $R^+$. 
Then we have the following:
\begin{enumerate}
    \item  $\Sigma_\uptau \coloneqq \{\uptau_{S/R}: S \in \mathcal{C}\}$ is a finite set of radical
    ideals of $R$.
    \item The splinter locus $\Spl(R)$ of $\Spec(R)$ is open and its
    complement is $\mathbf{V}(\uptau_R) = \mathbf{V}(\uptau^{\s}_R)$.
    \item $\uptau_R$ and $\uptau^{\s}_R$ are radical ideals and $\uptau_R = \uptau^{\s}_R$.
    \item There exists a finite generically \'etale extension domain $S$ of $R$
    such that $\uptau_R = \uptau^{\s}_R = \uptau_{S/R}$.
    \item There exists a finite generically \'etale extension domain $S$ of $R$
    such that if $R \hookrightarrow S$ splits, then $R$ is a splinter.
    \item If R is complete local,  
    there exists a finite generically \'etale $R$-subalgebra $S$ 
    of $R^+$
    such that
    $
    \uptau_{R^+/R} = \uptau_R = \uptau^{\s}_R = \uptau_{(R^+)^{\s}/R} =
    \uptau_{S/R}.
    $
\end{enumerate}
\end{theorem}

\begin{proof}
If $R$ satisfies (i), (ii) or (iii), we first claim that $R$ is approximately Gorenstein and \emph{N-1}.

By definition, $R$ is approximately Gorenstein if $R_\mathfrak{m}$ is 
approximately Gorenstein for all maximal ideals $\mathfrak{m}$ of $R$. If $R$ is
$F$-pure, so is $R_{\mathfrak m}$. Then $R_{\mathfrak m}$ is approximately Gorenstein
by \cite[Cor.\ 3.6(ii)]{DM19}, where it is shown more generally that $F$-injective
noetherian rings are approximately Gorenstein.

A quasi-excellent domain is Nagata \cite[\href{https://stacks.math.columbia.edu/tag/07QV}{Tag 07QV}]{stacks-project}, hence \emph{N-1}. Thus, the $\emph{N-1}$ property follows when $R$ is
$F$-finite because $F$-finite rings are excellent \cite[Thm.\ 2.5]{Kun76}. 
If $(A, \mathfrak m)$ has geometrically regular formal fibers, then $A$ is quasi-excellent
\cite[(33.D), Thm.\ 3.6]{Mat80} (or see \cite[Prop.\ 5.5.1]{ILO14}). Therefore if $R$ is essentially of finite type
over $A$, then $R$ is also quasi-excellent 
\cite[\href{https://stacks.math.columbia.edu/tag/07QU}{Tag 07QU}]{stacks-project}, hence \emph{N-1}. Finally, if $(R,\mathfrak m)$
is local and $F$-pure, then $R$ is \emph{N-1} by
\cite[Lem.\ 7.1.4]{DMS20}.

(1) An $F$-finite noetherian $F$-pure ring is Frobenius split. 
By Proposition \ref{prop:finiteness-F-compatible},
the collection of uniformly $F$-compatible ideals is finite in 
case (i) and also in case (ii) when $(R, \mathfrak m)$ is additionally Frobenius split, and by Lemma \ref{lem:F-compatible-properties},
every uniformly $F$-compatible is radical when $R$ is Frobenius split. 
Let $\mathcal C$ be the collection of finite
$R$-subalgebras of $R^+$ and $\mathcal{C}^{\s}$ be the subset of $\mathcal C$ consisting
of those $R$-subalgebras that are also generically \'etale. Then
\[
\Sigma_\uptau \coloneqq \{\uptau_{S/R} \colon S \in \mathcal{C}\}
\]
and
\[
\Sigma^{\s}_\uptau \coloneqq \{\uptau_{S/R} \colon S \in \mathcal{C}^{\s}\}
\]
are collections of uniformly $F$-compatible ideals by Lemma \ref{lem:trace-ideals}. In particular,
both $\Sigma$ and $\Sigma^{\s}$ are finite sets of radical ideals 
in case (i) and in case (ii)
when $(R, \mathfrak m)$ is Frobenius split. We will now show that 
$\Sigma_\uptau$ (and hence $\Sigma^{\s}_\uptau$) is also a finite set of radical
ideals in case (iii) and also in case (ii) when we drop the hypothesis that $(R, \mathfrak m)$ is Frobenius split. Note that in the
generality of (ii) and (iii), and $F$-pure noetherian domain $R$ need not be Frobenius split (or admit any 
nonzero $R$-linear map $F_*R \rightarrow R$) 
\cite{DM20},
so it is not at all obvious that $\Sigma_\uptau$ is finite or that its elements are radical ideals.

If $R$ is essentially of finite type over a local $G$-ring $(A, \mathfrak m)$, then
it suffices to show that ther eis a faithfully flat map $R \rightarrow R'$ such that
$R'$ is $F$-finite and Frobenius split. Indeed, suppose one can find such a cover
of $R$. Then for any $\uptau_{S/R} \in \Sigma_\uptau$,
\[
\uptau_{S/R}R' = \uptau_{S \otimes_R R'/R'}
\]
by Lemma \ref{lem:trace-ideals}(4) because $S$ is a finite extension of $R$. Thus, 
\[
\{\uptau_{S/R}R' : S \in \mathcal{C}\}
\]
is a set of uniformly $F$-compatible ideals of the Frobenius split
$F$-finite ring $R'$ because all the expansion ideals are traces. 
Therefore this set is finite and each $\uptau_{S/R}R'$
is a radical ideal by the argument in the previous paragraph. 
Since $R \rightarrow R'$ is faithfully flat,
\[
\uptau_{S/R} = \uptau_{S/R}R' \cap R.
\]
As contractions of radical ideals are radical, $\Sigma_\uptau$
must be a finite set of radical ideals as well.

We now show the existence $R'$. Let $\widehat{A}$ be the $\mathfrak{m}$-adic 
completion of
$A$. By our assumption, $A \rightarrow \widehat{A}$ is a regular map. Since
$R$ is essentially of finite type over $A$, by 
\cite[\href{https://stacks.math.columbia.edu/tag/07C1}{Tag 07C1}]{stacks-project}
and the fact that property of being regular is preserved under localization,
\[
R \rightarrow R \otimes_A \widehat{A}
\]
is also a regular map. Therefore the relative Frobenius
\[
F_*R \otimes_R (R \otimes_A \widehat{A}) \rightarrow F_*(R \otimes_A \widehat{A})
\]
is faithfully flat by results of Radu \cite[Thm.\ 4]{Rad92} and Andr{\'e} 
\cite[Thm.\ 1]{And93}. Since $R \rightarrow F_*R$
is pure, 
by base change
\[
R \otimes_A \widehat{A} \rightarrow F_*R \otimes_R (R \otimes_A \widehat{A})
\]
is pure, hence so is the composition
\[
R \otimes_A \widehat{A} \rightarrow F_*R \otimes_R (R \otimes_A \widehat{A})
\rightarrow F_*(R \otimes_A \widehat{A}).
\]
This last map is the Frobenius of $R \otimes_A \widehat{A}$. Thus, $R \otimes_A \widehat{A}$
is an $F$-pure ring which is essentially
of finite type over a complete local ring of prime characteristic
(see also \cite[Sec.\ 2]{Has10} for a generalization of this argument). 
Therefore by the gamma
construction, there exists a faithfully flat local map
\[
\widehat{A} \rightarrow \widehat{A}^\Gamma
\]
such that $\widehat{A}^\Gamma$ is $F$-finite and $R \otimes_A \widehat{A}^\Gamma = 
(R \otimes_A \widehat{A}) \otimes_{\widehat{A}} \widehat{A}^\Gamma$ is $F$-pure
\cite[Thm.\ 3.4(ii)]{Mur19}. Consequently, $R \otimes_A \widehat{A}^\Gamma$ is 
Frobenius split since it is $F$-finite. Then we can take 
$R' = R \otimes_A \widehat{A}^\Gamma$
to be the faithfully flat $F$-finite cover of $R$ that is Frobenius split.

It remains to show that if $(R, \mathfrak m)$ is a local $F$-pure ring that is not Frobenius split, then $\Sigma_{\uptau}$ is a finite set of radical ideals. The strategy
is similar to the one above for case (iii). Since $R$ is $F$-pure, so is its completion
$\widehat{R}$ \cite[Cor.\ 6.13]{HR74} (without any restrictions on the formal fibers
of $R$). Then $\widehat{R}$ is Frobenius split because $F$-purity and Frobenius
splitting coincide for complete local rings. We now have that $\widehat{R}$ has 
finitely many uniformly $F$-compatible ideals by Proposition \ref{prop:finiteness-F-compatible}, all of which are radical by Lemma \ref{lem:F-compatible-properties} because we have a splitting. 
Then, as above, $\Sigma_\uptau$ is a finite set
of radical ideals of $R$ because the expansions of these ideals in 
$\widehat{R}$ (which is faithfully flat over $R$) are again trace
ideals, and hence uniformly $F$-compatible, radical and finite in number.

Since $\Sigma_\uptau^{\s} \subseteq \Sigma_\uptau$, we have shown that if $R$
satisfies (i), (ii) or (iii), then $\Sigma_\uptau$ (hence also $\Sigma^{\s}_\uptau$)
is a finite set of radical ideals of $R$. This proves (1).

(2) By (1), Lemma \ref{lem:minimal-trace} and Remark \ref{rem:separable-trace-observations}
we conclude that $\Sigma_\uptau$ and $\Sigma^{\s}_\uptau$ have smallest elements under inclusion.
We can then apply Proposition \ref{prop:splinter-ideal-properties}(9) and 
Proposition \ref{prop:sep-splinter-ideal-properties}(9) to conclude that 
\[
\Spec(R) \setminus \mathbf{V}(\uptau_R) = \Spl(R)  =  \Spec(R) \setminus \mathbf{V}(\uptau^{\s}_R).
\]
This proves (2). 

There are two ways to prove (3). 
We have already observed that $R$ is \emph{N-1} and approximately
Gorenstein if it satisfies (i),
(ii) or (iii). Then we can apply Corollary \ref{cor:the-two-traces} to get (2).

Alternatively, both $\uptau_R$ and $\uptau^{\s}_R$ are radical ideals 
(they are intersections of ideals in $\Sigma_\uptau$) 
that define the non-splinter locus of $\Spec(R)$ by (2), so they must be
equal.

(4) follows by (3) and  Proposition \ref{prop:sep-splinter-ideal-properties}(5). 

For (5)  
choose an $S$ that satisfies the conclusion of (4). If $R \hookrightarrow S$
splits, we have $\uptau_R = \uptau^{\s}_R = R$, that is $R$ is a splinter 
(Proposition \ref{prop:sep-splinter-ideal-properties}(1)).

(6) The equalities
\[
\textrm{$\uptau_{R^+/R} = \uptau_R$ \hspace{2mm} and \hspace{2mm} 
$\uptau_{(R^+)^{\s}/R} = \uptau^{\s}_R$}
\]
follow by Proposition \ref{prop:splinter-ideal-properties}(4) and 
Proposition \ref{prop:sep-splinter-ideal-properties}(4). We are then done by (4).
\end{proof}

\begin{remark}
\label{rem:weaking-formal-fibers}
In Theorem \ref{thm:splinter-F-split}(iii), the formal fibers of 
$(A, \mathfrak m)$ are assumed to be geometrically regular 
in order to ensure that if $R$ is an essentially of finite type $A$-algebra
that is $F$-pure, then the base change 
$R_{\widehat{A}} \coloneqq R \otimes_A \widehat{A}$
is also $F$-pure. However, one can get
by with weaker assumptions on the formal fibers of $A$ in order to
get $F$-purity of $R_{\widehat{A}}$ from that of $R$, which is all
that is needed to construct a faithfully flat $F$-finite cover of $R$ that
is Frobenius split. One such condition is discussed in the present remark,
and another condition will be discussed in Remark \ref{rem:F-pure-homomorphisms}. 
Define a noetherian algebra $R$ over a field $k$ of characteristic $p > 0$ to be
\emph{geometrically $F$-pure} if for all finite purely inseparable 
extensions $\ell$ of $k$, $R \otimes_k \ell$ is $F$-pure. 

\noindent We now
claim that if $(A, \mathfrak m)$ is a noetherian local ring whose
formal fibers are Gorenstein and geometrically $F$-pure, then for
an essentially of finite type $A$-algebra $R$, if $R$ is $F$-pure
then so is $R_{\widehat{A}}$. Moreover, if $R$ is a domain then it is
\emph{N-1}.
We briefly indicate the steps needed to 
prove this result, following the strategy of \cite{DM19} that 
shows an analogous result for `Cohen--Macaulay and
geometrically $F$-injective'. The techniques in \cite{DM19} are 
inspired by arguments of V\'elez \cite{V\'el95}.
\begin{enumerate}
    \item It is known that if $(S, \mathfrak m)
    \rightarrow (T, \mathfrak n)$ is a flat local homomorphism
    of noetherian local rings
    whose closed fiber is Gorenstein and $F$-pure, then 
    $F$-purity ascends from $S$ to $T$. This is proved in
    \cite[Prop.\ 3.3]{Abe01} when $R$ and $S$ are $F$-finite, 
    and the general case appears in \cite[Thm.\ 7.3]{MP21}.

    
    \vspace{1mm}
    
    \item By (1) it suffices to show that if the formal fibers of $A$
    are Gorenstein and geometrically $F$-pure, then 
    $R \rightarrow R_{\widehat{A}}$ has Gorenstein and $F$-pure 
    fibers.
    
    \vspace{1mm}
    
    \item Let $k$ be a field of characteristic $p > 0$. If $R$
    is a noetherian $k$-algebra that is Gorenstein and 
    geometrically $F$-pure, then we claim that for all finitely
    generated field extensions $k'$ of $k$, $R_{k'} = R \otimes_k k'$ is
    Gorenstein and $F$-pure. By
    \cite[\href{https://stacks.math.columbia.edu/tag/0C03}{Tag 0C03}]{stacks-project}, 
    the Gorenstein property is preserved by
    base change along finitely generated field extensions. Thus, 
    $R_{k'}$ is Gorenstein. 
    Since $F$-purity satisfies faithfully flat descent, by the proof strategy of 
    \cite[Prop.\ 4.10]{DM19} and 
    \cite[Lem.\ 4.9]{DM19}, it suffices to show $R_{k'}$ is $F$-pure
    when $k'$ is a finitely generated separable extension and when
    $k'$ is a finite purely inseparable extension. If $k'$ is a finitely generated 
    separable extension of $k$, then $R \rightarrow R_{k'}$ is a
    regular homomorphism, so $R_{k'}$ is $F$-pure by (1) or the 
    argument in the proof of Theorem \ref{thm:splinter-F-split}(1). 
    If $k'$ is a finite purely inseparable extension of $k$, then
    $R_{k'}$ is $F$-pure by the definition of geometrically 
    $F$-pure.
    
    \vspace{1mm}
    
    \item We now claim that if $S \rightarrow T$ is a homomorphism
    of noetherian rings whose fibers are Gorenstein and 
    geometrically $F$-pure, then for every essentially of finite 
    type $S$-algebra $R$, the fibers of $R \rightarrow R \otimes_S T$
    are also Gorenstein and $F$-pure. The proof of this reduces to
    (3) by \cite[Lem.\ 7.3.7]{EGAIV_2}. Finally, (2) 
    follows by (4) upon taking $S = A$ and $T = \widehat{A}$. 
    Therefore if $R$ is $F$-pure, so is $R_{\widehat{A}}$.
    
    \vspace{1mm}
    
    \item  The formal fibers of $A$ are geometrically reduced since they are
    geometrically $F$-pure. By the Zariski-Nagata theorem
    \cite[Thm.\ 7.6.4]{EGAIV_2}, $A$ is a Nagata ring, that is, for all
    prime ideals $\mathfrak p$ of $A$, $A/\mathfrak{p}$ is a Japanese ring.
    Then $A$ is universally Japanese by \cite[\href{https://stacks.math.columbia.edu/tag/0334}{Tag 0334}]{stacks-project}.
    Hence $R$ is universally Japanese by
    \cite[\href{https://stacks.math.columbia.edu/tag/032S}{Tag 032S}]{stacks-project},
    and so, $R$ is \emph{N-1} if it is a domain.
\end{enumerate}
\end{remark}

\begin{remark}
    \label{rem:F-pure-homomorphisms}
    The proof of Theorem \ref{thm:splinter-F-split} shows, more generally,
    that if $R$ is a noetherian $F$-pure domain that admits a faithfully flat
    cover $R \to S$, where $S$ has finitely many trace ideals of finite extensions
    of $S$, then the splinter locus of
    $R$ is open in $\Spec(R)$. 
    To illustrate the utility of this observation, suppose 
    $(A, \mathfrak m) \xrightarrow{\varphi} R$ is an essentially of finite
    type $F$-pure homomorphism of noetherian
    rings in
    the sense of Hashimoto \cite[(2.3)]{Has10}, where $(A, \mathfrak m)$
    is a noetherian $F$-pure local ring and $R$ is a domain. Being an
    $F$-pure homomorphism means that 
    the relative Frobenius 
    \[F_{R/A} \colon F_*A \otimes_R R \to F_*R\]
    is a pure ring map. For example, $\varphi$ is $F$-pure if 
    $F_{R/A}$ is
    faithfully flat, or equivalently, if $\varphi$ is
    geometrically regular by \cite{Rad92, And93}.
    The hypothesis that $\varphi$ is $F$-pure
    and $A$ is $F$-pure implies that $R$ is also $F$-pure 
    \cite[Prop. 2.4(4)]{Has10}. Consider the commutative diagram
    \[
        \begin{tikzcd}
            A \arrow[r] \arrow[d, "\varphi"]
              & \widehat{A} \arrow[r] \arrow[d, " \varphi \otimes \widehat{A}"]  
              &\widehat{A}^\Gamma \arrow[d, " \varphi \otimes \widehat{A}^{\Gamma}"]\\
            R \arrow[r]
          &   R \otimes_A \widehat{A} \arrow[r] & R \otimes_A \widehat{A}^{\Gamma},
        \end{tikzcd}  
    \]
    where $\widehat{A}$ is $F$-pure by Remark \ref{rem:weaking-formal-fibers}(1), and $\widehat{{A}}^{\Gamma}$
    is chosen to be noetherian local, $F$-finite and $F$-pure 
    (equivalently, Frobenius split)
    via the $\Gamma$-construction. Then $\varphi \otimes \widehat{A}^\Gamma$
    is also essentially of finite type, and so, $R \otimes_A \widehat{A}^\Gamma$
    is noetherian and $F$-finite because $\widehat{A}^\Gamma$ is. Since
    $F$-pure homomorphisms are stable under arbitrary base change
    \cite[Prop. 2.4(7)]{Has10}, it follows that 
    $\varphi \otimes \widehat{A}^\Gamma$
    is $F$-pure. Consequently, $R \otimes_A \widehat{A}^\Gamma$
    is a faithfully flat cover of $R$ that
    is noetherian, $F$-finite and $F$-pure, 
    and so, it has finitely many trace ideals by 
    Proposition \ref{prop:finiteness-F-compatible} and Lemma 
    \ref{lem:trace-ideals}(2).
    Thus, the splinter locus of $\Spec(R)$ is open.
\end{remark}

The formation of trace commutes with Henselizations
and completions under certain conditions, giving
a refinement of Proposition \ref{prop:trace-henselization-completion}.

\begin{corollary}
\label{cor:trace-commutes-completion-henselization}
Let $(R, \mathfrak m)$ be a noetherian local domain which is $F$-pure and normal.
\begin{enumerate}
    \item If $R$ is Frobenius split,
    then $\uptau_RR^h = \uptau_{R^h}$.
    \item If $R$ has geometrically regular formal fibers, then 
    $\uptau_RR^h = \uptau_{R^h}$.
    \item If $R$ has geometrically regular formal fibers, then
    $\uptau_R\widehat{R} = \uptau_{\widehat{R}}$.
\end{enumerate}
\end{corollary}

\begin{proof}
(1) $R^h$ is a normal domain by 
\cite[\href{https://stacks.math.columbia.edu/tag/06DI}{Tag 06DI}]{stacks-project}. 
We claim that $R^h$ is also Frobenius split.
Indeed, since $R^h$ is a filtered colimit of \'etale $R$-algebras,
the relative Frobenius 
\[
F_*R \otimes_R R^h \rightarrow F_*R^h
\]
is an isomorphism because $R \to R^h$ is weakly \'etale 
\cite[\href{https://stacks.math.columbia.edu/tag/092N}{Tag 092N}]{stacks-project} 
and the relative Frobenius of a weakly \'etale map is an isomorphism
\cite[\href{https://stacks.math.columbia.edu/tag/0F6W}{Tag 0F6W}]{stacks-project}. 
Then $R^h$ is Frobenius split because splittings 
are preserved under base change. By 
Theorem \ref{thm:splinter-F-split}(1), $\uptau_R$ (resp. $\uptau_{R^h}$)
defines the non-splinter locus of $\Spec(R)$ (resp. $\Spec(R^h)$). 
Moreover, $\uptau_R$
and $\uptau_{R^h}$ are radical by Theorem \ref{thm:splinter-F-split}(3). 
Since $R \rightarrow R^h$ is a regular homomorphism, 
\[
R/\uptau_R \rightarrow R^h/\uptau_RR^h
\]
is a regular homomorphism as well by finite type base change
\cite[\href{https://stacks.math.columbia.edu/tag/07C1}{Tag 07C1}]{stacks-project}. 
Then $R^h/\uptau_RR^h$ is reduced by 
\cite[\href{https://stacks.math.columbia.edu/tag/07QK}{Tag 07QK}]{stacks-project} because $R/\uptau_R$ is reduced. 
In other words, $\uptau_RR^h$ is a radical ideal of
$R^h$. Therefore to show that $\uptau_RR^h = \uptau_{R^h}$, it suffices
to check that $\uptau_RR^h$ defines the non-splinter locus of $\Spec(R^h)$
as well.

By Proposition \ref{prop:trace-henselization-completion}(1), we 
have
\[
\uptau_{R}R^h \subseteq \uptau_{R^h}.
\]
Therefore $\mathbf{V}(\uptau_{R^h}) \subseteq \mathbf{V}(\uptau_RR^h)$. 
On the other hand, if $\mathfrak q \in \mathbf{V}(\uptau_RR^h)$, then
\[
\uptau_R \subseteq \mathfrak{q} \cap R,
\]
that is, $R_{\mathfrak{q} \cap R}$ is not a splinter. However
\[
R_{\mathfrak{q} \cap R} \rightarrow (R^h)_{\mathfrak{q}}
\]
is faithfully flat (because $R \rightarrow R^h$ is), 
and the splinter property satisfies faithfully flat
descent. Consequently, $(R^h)_{\mathfrak q}$ cannot be a splinter, and so,
$\mathfrak{q} \in \mathbf{V}(\uptau_{R^h})$. This establishes the other
inclusion $\mathbf{V}(\uptau_RR^h) \subseteq \mathbf{V}(\uptau_{R^h})$,
completing the proof of (1).

(2) Since the relative Frobenius $F_*R \otimes_R R^h \rightarrow F_*R^h$
is an isomorphism, $R^h$ is also $F$-pure when $R$ is $F$-pure. 
Moreover, the formal fibers of
$R^h$ are also geometrically regular
\cite[Thm.\ 5.3(i)]{Gre76} because $A \rightarrow A^h$
is ind-\'etale and hence absolutely flat. Therefore by Theorem \ref{thm:splinter-F-split}(2)
the non-splinter locus of $R$ (resp. $R^h$) is defined by $\uptau_R$ (resp. 
$\uptau_{R^h}$). One can now use the same argument as in (1) to get (2).

(3) Note that $\widehat{R}$ is $F$-pure. This is true for the completion of
any $F$-pure noetherian local ring by
Remark \ref{rem:weaking-formal-fibers}(1), 
but in our setting this also 
follows by the regularity of $R \rightarrow \widehat{R}$
and the Radu-Andr\'e theorem using the argument given in the proof of Theorem
\ref{thm:splinter-F-split}(1).
Proposition 
\ref{prop:trace-henselization-completion}(2) shows that 
$\uptau_{R}\widehat{R} \subseteq \uptau_{\widehat{R}}$ 
and Theorem \ref{thm:splinter-F-split}(2) shows that 
$\uptau_R$ (resp. $\uptau_{\widehat{R}}$) defines the non-splinter locus of $R$
(resp. $\widehat{R}$). One can then mimic
the argument of (1) to prove (3). We omit the details.
\end{proof}

\begin{remark}
The proof of Corollary \ref{cor:trace-commutes-completion-henselization}
shows more generally 
that if $(A, \mathfrak m)$ is a noetherian normal domain of 
arbitrary characteristic that satisfies the hypotheses of Proposition
\ref{prop:trace-henselization-completion}, and if $\uptau_A$, 
$\uptau_{A^h}$ and $\uptau_{\widehat{A}}$ define the non-splinter locus
$A$, $A^h$ and $\widehat{A}$ respectively, then $\uptau_AA^h$ and
$\uptau_{A^h}$ agree up to radical, as do $\uptau_A\widehat{A}$ and
$\uptau_{\widehat{A}}$.
\end{remark}

A natural question one can ask is whether $\uptau_R$ is nonzero. Indeed, if $\uptau_R$
is the ideal that defines the non-splinter locus in general, then $\uptau_R$ has to
be nonzero because a domain is generically a splinter. The next result implies that showing 
$\uptau_R \neq 0$ is equivalent to a long-standing conjecture
in tight closure theory on the existence of test elements.

\begin{proposition}
\label{prop:nonzero-splinter-ideal}
Let $R$ be a noetherian domain of 
characteristic $p > 0$ whose local
rings at maximal ideals have geometrically
regular formal fibers (i.e. $R$ is a $G$-ring) and 
whose regular locus is open. 
Let $\uptau^{\fg}_*(R)$ be the finitistic tight closure test ideal
of $R$. Then we have the following:
\begin{enumerate}
    \item $\uptau^{\fg}_*(R) \subseteq \uptau_R$.
     \item If $T_{F_*R/R} \neq 0$, then $\uptau^{\fg}_*(R) \neq 0$.
    \item $\uptau_R \neq 0$ if and only if 
    $\uptau^{\fg}_*(R) \neq 0$.
    \item If $R$ is $F$-pure, 
    then $\uptau_R \neq 0$.
    \item If $\uptau_{F_*R/R} \neq 0$, that is, if 
    there is a nonzero $R$-linear map $F_*R \rightarrow R$, then $\uptau_R \neq 0$.
\end{enumerate}
\end{proposition}

\begin{proof}
The completions of the local rings of $R$ at
maximal ideals are reduced because the property of being reduced is preserved 
under completion when the formal fibers are geometrically regular.
Thus,
$R$ is approximately Gorenstein.

(1) Since $R$ is approximately Gorenstein, $\uptau_R = T_{R^+/R} 
= \bigcap_I (I: IR^+ \cap R)$ by Proposition \ref{prop:splinter-ideal-properties}(3). For all ideals $I$ of $R$, 
we have
\[
IR^+ \cap R \subseteq I^*
\]
by \cite[Cor.\ (5.23)]{HH94(b)}. Thus, $(I:I^*) \subseteq (I:IR^+\cap R)$, and so,
\[
\uptau^{\fg}_*(R) = \bigcap_I (I:I^*) \subseteq \bigcap_I (I: IR^+ \cap R) =: \uptau_R.
\]
Here the first equality follows by \cite[Prop.\ (8.15)]{HH90} because $R$ is 
approximately Gorenstein. 
This proves (1). 

(2) Since
$R$ is a domain and the regular locus of $R$ is open, 
one can find a $c \neq 0$ such that $R_c$ is regular and 
$c \in T_{F_*R/R}$. Then the result follows by 
\cite[Thm.\ 1.2]{Abe93}.

(3) If $\uptau^{\fg}_*(R) \neq 0$, then $\uptau_R \neq 0$ by (1). 
Conversely, suppose $\uptau_R = T_{R^+/R} \neq 0$. Since
$F_*R$ embeds in $R^+$, this means that
\[
0 \neq T_{R^+/R} \subseteq T_{F_*R/R} = \bigcap_I (I: IF_*R\cap R).
\]
Then $\uptau^{\fg}_*(R) \neq 0$ by (2).

(4) follows (2) and (3) because if $R$ is $F$-pure, then
$T_{F_*R/R} = R$.

(5) If $\uptau_{F_*R/R} \neq 0$, then $T_{F_*R/R} \neq 0$
by Lemma \ref{lem:contraction-ideals} because 
$\uptau_{F_*R/R} \subseteq T_{F_*R/R}$. We are then 
done by
(2) and (3).
\end{proof}

We now deduce our main global result.

\begin{theorem}
\label{thm:F-split-locus-open}
Let $X$ be a scheme of prime characteristic $p > 0$. Suppose that $X$ satisfies any
of the following conditions:
\begin{enumerate}
    \item[(i)] $X$ is locally noetherian and $F$-finite.
    \item[(ii)] $X$ is locally essentially of finite type over a noetherian local ring
    $(A, \mathfrak m)$ of prime characteristic $p > 0$ with geometrically regular 
    formal fibers.
\end{enumerate}
Then
\[
\Spl(X) = \textrm{$\{x \in X \colon \mathcal{O}_{X,x}$ is a splinter$\}$}
\]
is open in $X$.
\end{theorem}

Recall that we say that $X$ is \emph{locally essentially of finite type over $A$} 
if there
exists an affine open cover $\Spec(B_i)$ of $X$ such that for all $i$, $B_i$
is an essentially of finite type $A$-algebra.

\begin{proof}
Let $\Nor(X)$ denote the normal locus of $X$ and $fp(X)$ denote the locus of points
$x \in X$ such that $\mathcal{O}_{X,x}$ is $F$-pure. We claim that both these loci
are open if $X$ satisfies (i) or (ii). Indeed, in either case $X$ is quasi-excellent
and has an open regular locus. Then $\Nor(X)$ is open by \cite[Cor. (6.13.5)]{EGAIV_2}.
If $X$ is locally noetherian and $F$-finite, then $fp(X)$ is open and coincides
with the locus of points at which $X$ is Frobenius split. 
If $X$ is locally essentially of
finite type over $(A, \mathfrak m)$ as in (ii), then $fp(X)$ is open by
\cite[Cor.\ 3.5]{Mur19} (this result is attributed to Hoshi when
$A$ is excellent in \cite[Thm.\ 3.2]{Has10b}).
Note that \cite[Cor.\ 3.5]{Mur19} is stated assuming $X$ is quasi-compact,
but for openness of loci, one can always work on an affine cover of $X$.

If $\mathcal{O}_{X,x}$
is a splinter, then it is normal and $F$-pure (the latter follows because the Frobenius
$F: \mathcal{O}_{X,x} \rightarrow F_*\mathcal{O}_{X,x}$ is an integral extension). Thus,
\[
\Spl(X) \subseteq \Nor(X) \cap fp(X).
\]
Therefore, replacing $X$ by $\Nor(X) \cap fp(X)$, we may assume $X$ is locally $F$-pure
and normal. Now there exists an affine open cover $\{\Spec(R_{\alpha})\}_\alpha$ of $X$,
where each $R_\alpha$ is an $F$-pure domain that satisfies 
condition (i) or (iii) of Theorem \ref{thm:splinter-F-split} depending on whether
$X$ satisfies conditions (i) or (ii) in the statement of this theorem
(you can even choose $R_\alpha$
to be normal). Then for all $\alpha$, 
$\Spl(X) \cap \Spec(R_\alpha) = \Spl(R_\alpha)$ is open in 
$\Spec(R_\alpha)$ by
Theorem \ref{thm:splinter-F-split}(2), and hence, also in $X$. Then 
$\Spl(X) = \bigcup_{\alpha} \Spl(R_\alpha)$ is open in $X$ as well.
\end{proof}

\begin{remark}
\label{rem:open-F-pure-loci}
In Remark \ref{rem:weaking-formal-fibers} we showed that 
Theorem \ref{thm:splinter-F-split} still holds if in part (iii) of
loc. cit. we assume that the formal fibers of $A$ are Gorenstein
and geometrically $F$-pure. We claim that the same hypotheses
on the formal fibers of $A$ also work for Theorem \ref{thm:F-split-locus-open}. 
The two things 
we need to check are: 
\begin{enumerate}
    \item \emph{If $A$ has Gorenstein and geometrically
    $F$-pure formal fibers, then for any essentially of finite type
    $A$-algebra $R$, the $F$-pure locus of $R$ is open}:
    We follow the proof strategy of \cite[Thm.\ B]{DM19} which establishes
    the analogous fact for the property `Cohen--Macaulay and
    geometrically $F$-injective'. Since the 
    induced map $\Spec(R_{\widehat{A}}) \rightarrow \Spec(R)$
    is faithfully flat and quasi-compact, by \cite[Cor.\ 2.3.12]{EGAIV_2}, 
    it suffices to show that the inverse image of
    the $F$-pure locus of $\Spec(R)$ is open in $\Spec(R_{\widehat{A}})$. 
    But this inverse
    image is the $F$-pure locus of $\Spec(R_{\widehat{A}})$ by
    Remark \ref{rem:weaking-formal-fibers}(1),(4) and by 
    faithfully flat 
    descent of $F$-purity. That the $F$-pure locus of $R_{\widehat{A}}$
    is open now follows by \cite[Cor.\ 3.5]{Mur19} because $\widehat{A}$ 
    is excellent. 
    
    \vspace{1mm}

    \item \emph{If $A$ has Gorenstein and geometrically
    $F$-pure formal fibers, then the normal locus of any essentially of finite type
    $A$-algebra $R$ is open}:
    By Remark \ref{rem:weaking-formal-fibers}(5), $R$ is universally Japanese,
    and so, for all finite $R$-algebras $S$ such that $S$ is domain, 
    the normal locus of $S$ is open by 
    \cite[Cor.\ 6.13.3]{EGAIV_2}. Then the normal locus of $R$ is open by
    \cite[Prop.\ 6.13.7]{EGAIV_2}.
\end{enumerate}
\end{remark}

\begin{example}
In order to show the openness of splinter loci in prime characteristic, it suffices to restrict ones attention to the intersection of the $F$-pure and normal loci. This intersection 
is open as long as the $F$-pure locus and the normal locus are both open. 
Furthermore, we have
shown that in the local case, a noetherian $F$-pure domain always has an open
splinter locus (Theorem \ref{thm:splinter-F-split}). Thus one may naturally wonder if
the splinter loci of an $F$-pure and normal noetherian domain is always open. We now
use a construction of Hochster \cite{Hoc73(b)} to give examples of locally excellent $F$-pure and normal domains of prime characteristic $p > 0$ whose splinter loci are not
open. 
We begin by choosing an algebraically closed field $k$ of prime characteristic
$p > 0$ and a local domain $(R, \mathfrak m)$ essentially of finite type over $k$
such that $(R, \mathfrak m)$ is $F$-pure and normal, $R$ is not a splinter and $R/\mathfrak m = k$. The last hypothesis ensures that if $K/k$ is any field extension, then $\mathfrak{m}(K \otimes_k R)$ is a maximal ideal of $K \otimes_k R$. For an explicit example,
if $k$ is a field of characteristic not equal to $3$, then the local ring 
at the origin of the Fermat cubic
\[
R = k[x,y,z]_{(x,y,z)}/(x^3+y^3+z^3)
\]
is not $F$-rational \cite[Ex.\ 6.2.5]{Smi97(b)}, hence
also not a splinter since excellent splinters are $F$-rational \cite{Smi94}.
By Fedder's criterion, $R$ is $F$-pure, for instance, when the characteristic of $k$ is
$7$ and $R$ is normal (it is $R_1$ + $S_2$) when the characteristic of $k \neq 3$. Coming
back to our example, once we have such an $R$, Hochster then constructs \cite[Prop.\ 1]{Hoc73(b)} a
noetherian domain $S$ using $R$ such that
\begin{enumerate}
    \item[(a)] $S$ has infinitely many maximal ideals;
    \item[(b)] for any maximal ideal $\mathfrak M$ of $S$,
    \[
    S_{\mathfrak M} \cong (L_{\mathfrak M} \otimes_k R)_{\mathfrak{m}(L_{\mathfrak M} \otimes_k R)}
     \]
     for a suitable field extension $L_{\mathfrak M}/k$ that depends on $\mathfrak{M}$;
    \item[(c)] every nonzero element of $S$
    is contained in only finitely many maximal ideals.
\end{enumerate}
  
In particular, this implies that $S$ is a 
locally excellent domain; in fact the local rings of $S$ are essentially of finite type over appropriate field extensions of $k$. Furthermore,
(a) and (c) imply that the intersection of all the maximal ideals of $S$ is $(0)$.

We now claim that since $k$ is algebraically closed, for all field extensions $K/k$, 
$K \otimes_k R$ is also $F$-pure and normal. Indeed, since $k$ does not have any non-trivial finite purely inseparable extensions, 
\[
k \hookrightarrow K
\]
is trivially a regular map (i.e. a flat map with geometrically regular fibers). As $R$ is
essentially of finite type over $k$, the base change map
\[
R \to K \otimes_k R
\]
is also a regular map \cite[(33.D), Lem.\ 4]{Mat80}. We now observe that both $F$-purity
and normality ascend from the base to the target over regular maps, proving that $K \otimes_k R$ is both $F$-pure
and normal. The ascent of $F$-purity over regular maps follows by the
Radu-Andr\'e theorem because the relative Frobenius of $F_{K \otimes_k R/R}$ is 
faithfully flat, hence pure, and so the Frobenius on $K \otimes_k R$ can be expressed as 
the composition of the pure maps
\[
K \otimes_k R \xrightarrow{\id_K \otimes_k F_R} K \otimes_k F_*R \xrightarrow{F_{K \otimes_k R/R}} F_*(K\otimes_k R),
\]
where the first map in the composition is pure because it is the base change of the 
pure map $F_R \colon R \to F_*R$. The ascent of normality over regular maps follows
because the $R_n$ and $S_n$ properties ascend over regular maps; see for instance
\cite[\href{https://stacks.math.columbia.edu/tag/0BFK}{Tag 0BFK}]{stacks-project}. The
upshot of this discussion is that 
for the locally excellent noetherian domain $S$ and for
any maximal ideal $\mathfrak M$ of $S$, $S_{\mathfrak M} \cong (L_{\mathfrak M} \otimes_k R)_{\mathfrak{m}(L_{\mathfrak M} \otimes_k R)}$ is both $F$-pure and normal. Since $F$-purity and normality can be checked locally at the maximal ideals, it follows that $S$ is a locally excellent $F$-pure and normal domain. However,
$R$ was chosen so that it is not a splinter, and the splinter property
satisfies faithfully-flat descent. Therefore for all maximal
ideals $\mathfrak{M}$ of $S$, $S_{\mathfrak M}$ is not a splinter
because the map $R \to (L_{\mathfrak M} \otimes_k R)_{\mathfrak{m}(L_{\mathfrak M} \otimes_k R)} \cong S_{\mathfrak M}$ is faithfully flat. Thus the non-splinter locus of $S$ contains
all the maximal ideals, whose intersection is $(0)$. This means that the non-splinter locus of $S$ cannot be closed as otherwise the splinter locus
would be empty, which it is not -- $S$, being a domain, is a splinter at its generic point.
\end{example}


For noetherian graded rings over fields of 
prime characteristic, the
splinter property is detected by the homogeneous maximal ideal. This 
is well-known over fields of characteristic $0$ because then splinter
is the same as being normal.

\begin{corollary}
\label{cor:splinter-graded}
Let $R = \bigoplus_{n = 0}^{\infty} R_n$ be a noetherian graded ring such
that $R_0 = k$ is a field of characteristic $p > 0$.
Let $\mathfrak{m} \coloneqq \bigoplus_{n > 0} R_n$ be the homogeneous maximal
ideal of $R$. Then $R$ is a splinter if and only if $R_{\mathfrak m}$ is a
splinter.
\end{corollary}

\begin{proof}
The splinter property localizes. So the backward implication
is the non-trivial one.

We first assume that $k$ is infinite.
By Theorem \ref{thm:F-split-locus-open}, the splinter locus of $R$ is open.
Let $I$ be the radical ideal of $R$ that defines the non-splinter locus. 
Note that $k^\times$ acts on $R$ by automorphisms 
(for $c \in k^\times$ and $x \in R_n$, $c \cdot x = c^nx$), and automorphisms clearly preserve the non-splinter locus of $R$. 
This means that $I$ is stable under the
action of $k^\times$, and since $k$ is infinite, $I$ must be a homogeneous
ideal of $R$. If $R_{\mathfrak m}$ is a splinter, then $I \nsubseteq \mathfrak m$,
which means that $I = R$, that is, $R$ is a splinter.

Suppose $k$ is finite and that $R_\mathfrak{m}$ is a splinter. 
Finite fields of prime characteristic are
perfect. This means that an algebraic closure $\overline{k}$ is the
filtered union of finite \'etale subextensions $\ell/k$. 
Let
\[
R_\ell \coloneqq \ell \otimes_ k R.
\]
Then $R \hookrightarrow R_\ell$ is a finite \'etale map and 
$\mathfrak{m}_\ell \coloneqq \mathfrak{m}R_\ell$ is the homogeneous maximal ideal
of $R_\ell$. It follows that
\[
R_{\mathfrak m} \hookrightarrow (R_\ell)_{\mathfrak{m}_\ell}
\]
is essentially \'etale, and so, $(R_\ell)_{\mathfrak{m}_\ell}$ is a splinter
by \cite[Thm.\ A]{DT19}. Now
\[
R_{\overline{k}} \coloneqq \overline{k} \otimes_k R.
\]
is graded noetherian over $\overline{k}$ with homogeneous maximal ideal $\mathfrak{m}_{\overline{k}} \coloneqq \mathfrak{m}R_{\overline{k}}$.
Then
\[
(R_{\overline{k}})_{\mathfrak{m}_{\overline{k}}} = \colim_\ell (R_\ell)_{\mathfrak{m}_\ell}
\]
is a splinter because a filtered colimit of splinters is a splinter
\cite[Prop.\ 5.2.5(ii)]{AD20}. Since $\overline{k}$ is infinite, it
follows by the previous paragraph that $R_{\overline{k}}$ is a splinter. 
By faithfully flat descent along
\[
R \hookrightarrow R_{\overline{k}},
\]
we then get that $R$ is a splinter.
\end{proof}

\section{Acknowledgments}
We thank Linquan Ma, Takumi Murayama, Rebecca R.G. and Karl Schwede for helpful conversations and their
interest. We are very grateful to Takumi for allowing us
to include the results in Remarks \ref{rem:weaking-formal-fibers} and 
\ref{rem:open-F-pure-loci} on permanence properties of
Gorenstein and geometrically $F$-pure fibers and their relation
to openness of $F$-pure
loci since he was aware of these results from his
work in \cite{Mur19}. Additionally, we thank Linquan and Benjamin Antieau for 
comments on a draft.

\bibliographystyle{amsalpha}

\end{document}